\documentclass[a4paper,12pt]{article}
\usepackage{amssymb, amsmath, amsthm}
\usepackage{enumitem}
\theoremstyle{definition}
\newtheorem{Teo}{Theorem}
\newtheorem{Cor}{Corollary}
\newtheorem{remark}{Remark}
\newtheorem{Lem}{Lemma}
\newtheorem{Prop}{Proposition}

\newtheorem*{Lem1'}{Lemma 1'}
\newtheorem*{Lem2'}{Lemma 2'}
\numberwithin{equation}{section}

\title{Convergence at infinity for solutions of nonhomogeneous degenerate elliptic \\ equations in exterior domains}

\author{L. P. Bonorino \and L. P. Dutra \and F. J. dos Santos}

\date{}

\begin{document}

\maketitle

\begin{abstract}
\par In this work, first we prove that for any compact set $K\subset\mathbb{R}^{n}$ and any continuous function $\phi$ defined on $\partial K$, there exists a bounded weak solution in $C(\overline{\mathbb{R}^{n}\backslash K}) \cap C^1(\mathbb{R}^{n}\backslash K)$ to the exterior Dirichlet problem
$$
\begin{cases}
     -{\rm div}\big(\,|\nabla u|^{p-2}A(\,|\nabla u|\,)\nabla u\,\big)=f   &  \text{ in }\, \mathbb{R}^n \backslash K \\
    \;\;\;\;\;\; u = \phi &  \text{ on } \partial K\\
 \end{cases}
 $$
 provided $p > n$, $A$ satisfies some growth conditions and $f\in L^{\infty}(\mathbb{R}^n)$ meets a suitable pointwise decay rate. We obtain thereafter the existence of the limit at the infinity for solutions to this problem, for any $p\in(1,+\infty)$ and $n\geq2$. Moreover, for $p > n$ we can show that the solutions converge at some rate and for $p <n$ the convergence holds even for some unbounded $f$.
\end{abstract}

\section{Introduction}

In this paper we study two questions about an exterior Dirichlet problem for some nonlinear elliptic PDE. The first one is
the existence of a bounded weak solution in $C(\overline{\mathbb{R}^{n}\backslash K}) \cap C^1(\mathbb{R}^{n}\backslash K)$ for
\begin{equation}\label{nonhomogeneousPr}
  \begin{cases}
     -{\rm div}\big(\,|\nabla u|^{p-2}A(\,|\nabla u|\,)\nabla u\,\big)=f   &  \text{ in }\, \mathbb{R}^n \backslash K \\[5pt]
    \;\;\;\;\;\; u = \phi &  \text{ on } \partial K,\\
 \end{cases}
\end{equation}
where $K$ is any compact set, provided $A$ and $f$ satisfy some suitable conditions and $p > n$.
The second is the existence of the limit for any solution to this problem for $p >1$.

Initially we assume that $A$ satisfies: \\

\begin{equation}\label{A_conditions}
\begin{split}
&i) \,\,\,\, A \in C^1\big([0,+\infty]\big);\\
&ii) \,\,\, \delta\leq A(t) \leq L,\,\, \text{for all}\, \,t\ge0, \,\, \text{where}\,\, \delta, L \,\,\text{are positive constants;}\\
&iii) \,\,\, \delta'\,t^{\,p-2}\,\leq \,\frac{d}{dt}\,\Big\{\, t^{\,p-1}A(t)\,\Big\}\,\leq L'\,t^{\,p-2}\,, \,\,\text{for all}\, \,t\ge0, \,\,\text{where} \,\, \delta', L' \\
&\,\,\,\,\,\,\text{are positive constants.}
\end{split}
\end{equation}

Our first result states the existence and uniqueness of continuous bounded weak
solutions for Dirichlet problems on exterior domains in case $p > n$.

\begin{Teo}\label{TeorPrinc1}
Let $K\subset\mathbb{R}^{n}$ be a nonempty compact set and $\phi\in C(\partial K)$. Assume that the function $A$ satisfies \eqref{A_conditions} and $f\in L^{\infty}(\mathbb{R}^n)$ is such that, for positive constants $C_f$ and $ \epsilon$,
\begin{equation}\label{f}
|f(x)|\,\leq \,C_f\,|x|^{\,-(p+\epsilon)}\,
\end{equation}
for all $|x|$ sufficiently large. If $p>n$, then there exists a unique bounded weak solution $u \in C(\overline{\mathbb{R}^{n}\setminus K})\cap C^1(\mathbb{R}^{n}\setminus K)$ of \eqref{nonhomogeneousPr}.
In addition, if $\phi $ is $\alpha$-H\"{o}lder continuous in $K$, with $\alpha = \frac{p-n}{p-1}$, then $u\in C^{\alpha}(\mathbb{R}^n).$
\end{Teo}
\par We point out the full generality allowed for the boundary $\partial K$, for which no regularity has to be assumed.
This result also does not require that $\phi \in W^{1,p}_{loc}(\mathbb{R}^n)$, which would imply the existence of local $C^{0,1-n/p}$ solutions by the use of minimizing methods and Morrey's inequality (see \cite{Lq2}). Moreover, as it is straightforward from the proof, the result also holds for bounded domains, being necessary only to assume $f\in L^{\infty}(\mathbb{R}^n)$.
\par Many efforts were directed to elliptic problems on unbounded domains. For instance, Meyers and Serrin \cite{MS} have made important clarifications on the existence and uniqueness of bounded solutions of linear exterior problems, and some of their results were extended to several classes of semilinear equations by Kusano \cite{K}, Ogata \cite{O}, Noussair \cite{N1, N2}, Furusho et. al. \cite{Fu1, Fu2}, Phuong C\'ac \cite{Cac}, among others.

In our theorem, the regularity up to the boundary is proved with use of barriers following the ideas from Serrin \cite{S1}, where the main concern was the Liouville property for entire solutions of homogeneous problems. An extension of the Liouville property for exterior solutions was obtained by Bonorino et. al. \cite{BSZ}, which motivated our first result.
Indeed this one generalizes a result presented in \cite{BSZ} stating that for $p>n$ and for a finite set $P\subset\mathbb{R}^n$, there exists a bounded $p$-harmonic function in $\mathbb{R}^n\backslash P$ attaining any prescribed data in $P$.

\

\par Another question arising on exterior problems is the behavior of solutions at infinity, which is related to the theory of singularities of solutions. A pioneer in this field is the work of Gilbarg and Serrin \cite{G-S}, that shows the existence of the limit of solutions for nonhomogeneous linear elliptic equations on a singular point and also at infinity in the homogeneous case. The argument depends essentially on the Harnack inequality combined with a Maximum Principle. Later, a variety of results on removable singularities and the asymptotic behavior for several equations were obtained by Serrin \cite{Se2, Se3, Se4, Se5}, Serrin and Weinberger \cite{SW}, and others. \cite{Se3} presents a detailed description of the asymptotic behavior at the origin and at infinity of positive solutions of the homogeneous quasilinear equation ${\rm div}\mathcal{A}(x,Du)=0$. It was shown that positive solutions $u$ always converge at infinity to a possibly infinite limit $\ell$ and, moreover, either $u$ satisfies a {\em  maximum principle at infinity } or else $\ell$ is infinite if $p\ge n$ and finite if $p<n$, and it holds

\begin{equation*}
\begin{split}
u\,&\approx\,\,r^{\,(p-n)/(p-1)}\,,\,\,\,\,\,\,\,\,\,\,\,\,\;  p>n\\
u\,&\approx\,\,\log\,r \,,       \,\,\,\,\,\,\,\,\,\,\,\,\,\,\,\,\,\,\,\,\,\,\,\,\,\,\,\,\,  p=n \\
u-\ell\,&\approx\,\,\pm\, r^{\,(p-n)/(p-1)}\,, \,\,\,\,\,\,\,  p<n\\
\end{split}
\end{equation*}
where $u\approx v$ means that there exists positive constants $c, C$ such that $c\,v\le u \le C\,v$. The function $u$ is said to satisfy the maximum principle property at infinity if in any neighbourhood of infinity either $u$ is constant or else takes on values both greater and less than $\ell<+\infty$.
More recently, it was proved by Fraas and Pinchoverin in \cite{FrP2} the existence of limit near singularities for nonnegative solutions of
$$-\Delta_p\,u+V\,|u|^{p-2\,}u=0$$
assuming that near the singularity the potential $V$ belong to a \textit{Kato class} and
$$V\in L^{\infty}_{loc}\,,\,\,\,|x|^p\,|V(x)|\le C\,,\,\,\,\text{for some constant }C\,.$$

\

This motivates our second result, where we investigate the existence of the limit at infinity for bounded solutions $u\in C^{1}(\mathbb{R}^{n}\setminus K)$ of \eqref{nonhomogeneousPr}
for any $p >1$, assuming only the following hypotheses on $A$:
\begin{equation}\label{A_conditions2}
\begin{split}
&i)\,\,\, A \in C\big([0,+\infty]\big);\\
&ii) \,\,\,\delta\leq A\leq L\,,\,\, \text{for positive constants } \,\delta, L\,;\quad\quad\quad\quad\quad\quad\quad\quad\quad\quad\quad\quad\quad\quad\\
&iii) \,\,\,t\mapsto t^{\,p-1}A(t)\,\, \text{is strictly increasing for }\,t>0\,.\\
\end{split}
\end{equation}

We prove the existence of the limit at infinity for such solutions, for any $p\in(1,+\infty)$, provided $f$ satisfies some suitable conditions. In the case $p<n$, by employing Harnack inequality arguments like in \cite{G-S} and comparison results for spherical rearrangements established by Talenti \cite{T1, Talenti} and others, we can obtain a sufficient condition for the existence of the limit in terms of some integrability of $f$ at infinity. For $p>n$, requiring that $f$ satisfies the growth condition \eqref{f}, we give a barrier argument to show the existence of the limit at infinity and that the convergence has a positive order $\beta$. Indeed, for any $p >1$, we prove that \eqref{f} guarantees the existence of the limit at infinity of any bounded solution. \\

\par On the matter of the behavior of the solutions at infinity, we can assume with no loss of generality  $K=\overline{B_1}$, where $B_1$ is the open ball of radius $1$ centered at the origin.
 Our second theorem then reads as follows.

\begin{Teo}
   Let $u\in C^{1}(\mathbb{R}^{n}\backslash B_1)$ be a bounded weak solution of \eqref{nonhomogeneousPr} in $\mathbb{R}^{n}\backslash B_1$, where $A$ satisfies \eqref{A_conditions2}. Then the following hold:
\\[5pt] $(a)$ for $p >1$ the limit of $u$ at infinity exists provided $f$ satisfies \eqref{f}; \\
(In this case, assuming that $u$ is only bounded from above (or from below), then either $u$ converges at infinity or $\lim_{|x| \to +\infty} u(x) = -\infty \; (\text{or} +\infty)$;)
\\[5pt] $(b)$ for $1< p< n$ the limit of $u$ at infinity exists if, more generally, $f$ satisfies
  \begin{equation}\label{Lr}
  f\in L^{r}(\mathbb{R}^{n}\backslash B_1) \,,\,\,\text{ for some } \,\,r\,<\,n/p
  \end{equation}
  and
  \begin{equation}\label{Ltheta}
  f\in L^{\,\frac{n}{p-\theta}}(\mathbb{R}^n\backslash B_1) \,,\,\,\text{ for some }\,\,\theta\in(0,1)
   \end{equation}
  with
  \begin{equation}\label{Kgoes0}
  \lim_{R\rightarrow +\infty}\,\,\,R^{\theta}\,\|\,f\,\|_{L^{\,\frac{n}{p-\theta}}(\,\mathbb{R}^n\backslash B_R(0)\,)}\,=\,0;
  \end{equation}
\\$(c)$ for $p>n$, if $f$ satisfies \eqref{f}, there are positive constants $C, \beta$ such that
  $$ \big|\,\underset{\,|x|\rightarrow\infty}{\lim u}-u(x)\,\big| < C|x|^{-\beta} \quad \text{ for all } \quad |x| \quad {\rm large.} $$
Moreover, for $p\ge n$ these results also hold for unbounded weak solutions $u$ provided we assume that
\begin{equation}
  \lim_{x \to \infty} \frac{|u(x)|}{|x|^{\alpha}}=0 \quad {\rm where} \quad \alpha=\frac{p-n}{p-1} \quad \text{ if } \quad p > n
\label{alpha-subgrowth}
\end{equation}
or
\begin{equation}
  \lim_{x \to \infty} \frac{|u(x)|}{\log|x|}=0 \quad \text{ if } \quad p=n.
\label{alpha-subgrowth2}
\end{equation} \\
\label{TeorPrinc}
\end{Teo}
\noindent
\par\textbf{Remarks:} Condition \eqref{f} is a particular case of conditions \eqref{Lr}, \eqref{Ltheta}, \eqref{Kgoes0} if and only if $\epsilon>0$. In addition, the result is the best possible with respect to the exponent $-p-\epsilon<-p$ in \eqref{f}. In fact, the theorem is false if $\epsilon=0$, a counterexample being given by the function
$$u(x) = \cos\,(\,\log\,\log |x|\,)\,,\,\,\text{ for }\,\,|x|>1.$$
Clearly $u$ does not attain a limit at infinity and satisfies
$$|\Delta_p u(x)|\leq C\,(\log|x|)^{-p+1}\,|x|^{-p}\,,\,\,\text{ for all }\,\,|x|\ge 2$$
for some positive constant $C$. In this same example observe that, in the case $\frac{p}{p-1}<n$, $f$ satisfies conditions \eqref{Ltheta} and \eqref{Kgoes0}, $f\in L^{s}(\mathbb{R}^n\backslash B_2(0))$ for $s \ge n/p$, but fails to satisfy \eqref{Lr} for any $r<n/p$. Hence, in this setting, assuming that $f\in L^{r}(\mathbb{R}^{n}\backslash B_1)$, condition $r < n/p$ is optimal in \eqref{Lr}.

In this work, we present in Section 2 some preliminary concepts and results, that include a Harnack inequality, the construction of radially symmetric barriers and some estimates for the symmetrization of solutions. Theorem \ref{TeorPrinc1} is proved in Section 3 and Theorem \ref{TeorPrinc} is split into Section 4 and 5, since different techniques are used for the cases $p \ge n$ and $p < n$.

\section{Preliminaries}

\subsection{Existence results and comparison principle}

 \par We recall that a function $u\in W^{1,p}_{loc}(\Omega)$ on some domain $\Omega\subseteq\mathbb{R}^{n}$ is a weak solution of
\begin{equation}\label{main}
-{\rm div}\big(\,|\nabla u|^{p-2}A(|\nabla u|)\nabla u\,\big)\,=\,f \,\,\,  \text{ in }\, \Omega
\end{equation}
if $u$ satisfies
\begin{equation}\label{weak}
\int_{\Omega}|\nabla u|^{p-2}A(|\nabla u|)\nabla u \cdot \nabla \eta \, dx \,=\, \int_{\Omega}\,f\eta \, dx
\end{equation}
for all $\eta\in C^{\infty}_{0}(\Omega)$. Solutions of \eqref{main} correspond to the critical points of a functional $J$ whose Lagrangian is
$$L(q,z,x) = \int_{0}^{\,|q|}\varphi(t)\,dt\,-\,z\,f(x)\,,\quad\,(q,z,x)\in \mathbb{R}^n\times\mathbb{R}\times\Omega\,$$
where $\varphi(t)=t^{p-1}A(t)$. Assuming that $A$ satisfies \eqref{A_conditions}, it follows that $L$ is $C^2$ and convex in the variable $q$, with
$$L_{\,q_i\,q_j}(q,z,x)\,\xi_i\,\xi_j \,\ge\,\frac{1}{2} \min\,\big\{\,\delta\,|q|^{\,p-2},\,\varphi'(|q|)\,\big\}\,|\xi|^2\quad\text{for all}\,\,\xi\in\mathbb{R}^n,$$
and satisfies
$$L(q,z,x)\,\le\,C(\,|q|^p+|z|^p+1\,)$$
$$D_{z}\,L(q,z,x)\,\le\,C(\,|q|^{p-1}+|z|^{p-1}+1\,)$$
$$D_{q}\,L(q,z,x)\,\le\,C(\,|q|^{p-1}+|z|^{p-1}+1\,).$$
Furthermore, if $\Omega$ is bounded and $g\in W^{1,p}(\Omega)$, we have that the functional $J(v) = \int_{\Omega} L(Dv,v,x) \, dx$ is bounded from below on the class $$\mathcal{A}_g\,=\,\{\,u\in W^{1,p}(\Omega)\,\,|\,\,u-g\in W^{1,p}_0(\Omega)\,\}.$$
Therefore, by classical variational techniques, a minimizer $u\in \mathcal{A}_g$ for $J$ is granted to exist for any given $f\in L^{\infty}(\Omega)$ and $g\in W^{1,p}(\Omega)$. Moreover such minimizer is a solution of \eqref{main} and, if $\Omega$ is also smooth, then $u=g$ on $\partial \Omega$ in the trace sense.
These hypotheses on $\Omega$ also guarantee that $u$ is bounded. Indeed, for $p > n$, this a consequence of Morrey's Inequality and, for $p \le n$, the boundedness of $u$ can be found in Theorem 3.12 of \cite{MZ}, that is an application of Moser iteration.  Hence, condition \eqref{A_conditions} implies the $C^{1,\beta}$ regularity up to the boundary of $u$ (see Theorem 1 of \cite{Lieberman}), that can be stated as follows:

\begin{Teo}\label{existence1}
Suppose that $\Omega$ is a bounded domain of class $C^{1,\alpha}$, $\alpha>0$, and $A$ satisfies \eqref{A_conditions}. If $g\in C^{1,\alpha}(\overline{\Omega})$, there exists a unique weak solution $u\in C^{1,\beta}(\overline{\Omega})$, for some $\beta>0$, of \eqref{main} satisfying $u=g$ in $\partial\Omega$. The H\"{o}lder seminorm $|u|_{1+\beta}$ depends only on $|g|_{1+\alpha}$, $\,\alpha, \,n, \,p, \,|\Omega|, \,\sup_\Omega\,|u|$, and the parameters of the equation.
\end{Teo}
\par This result can be extended for the case that $\phi$ is a continuous boundary data. For that take a sequence of $C^{1,\alpha}$ functions $g_n$ that converges to $\phi$ uniformly and consider the respective solutions $u_n$ given by Theorem \ref{existence1}.  From the comparison principle as stated in Theorem \ref{ComparisonPrincipleForBounded}, $u_n$ converges uniformly to a continuous function $u$ in $\overline{\Omega}$, that is also a weak solution of \eqref{main}. According to local regularity results (for instance, see Theorem 1 of \cite{Tolksdorf}), $u \in C^1(\Omega)$ and $\nabla u$ is a H\"older continuous function on compact subsets of $\Omega$.

\begin{Teo}\label{existenceForContinuousBoundaryData}
Suppose that $\Omega$ is a bounded domain of class $C^{1,\alpha}$, $\alpha>0$, and $A$ satisfies \eqref{A_conditions}. If $\phi\in C(\partial{\Omega})$, there exists a unique weak solution $u\in C(\overline{\Omega}) \cap C^{1}(\Omega)$ of \eqref{main} satisfying $u=\phi$ in $\partial\Omega$. Moreover, $\nabla u$ is a H\"older continuous function on compact subsets of $\Omega$.
\end{Teo}

\par Along this work, we need the following comparison principle, that is a particular case of Theorem 2.4.1 of \cite{SerrinPucci}. Note that the vector function $\mathbf{A}(\xi)=|\xi|^{\,p-2}A(\,|\xi|\,)\,\xi \,$ satisfies the monotonicity condition
$$ \big(\mathbf{A}(\xi)-\mathbf{A}(\eta)\big)\cdot\big(\xi-\eta\big)\,>\, 0\,,\,\,\,\,\text{for all}\,\,\,\,\,\xi,\,\eta\in\mathbb{R}^n,\,\,\,\, \xi\neq\eta\,.$$

\begin{Teo}
Suppose that $A$ satisfies \eqref{A_conditions2} and $u$, $v\in C(\overline{\Omega})\cap W^{1,p}(\Omega)$ on a bounded domain $\Omega\subset \mathbb{R}^{n}$. If
$$-\text{div}\big(\,|\nabla u |^{p-2} A(\,|\nabla u|\,) \nabla u\,\big)\,\leq \,-\text{div}\big(\,|\nabla v |^{p-2} A(\,|\nabla v|\,) \nabla v\,\big)$$
in the weak sense and $u\leq v \mbox{ in }\partial \Omega$, then $u\leq v$ in $\Omega\,$.
\label{ComparisonPrincipleForBounded}
\end{Teo}

\subsection{Radially symmetric solutions for $p > n$}

\par  Next we establish some results on existence and estimates for barriers to the problem \eqref{nonhomogeneousPr} in order to prove Theorem \ref{TeorPrinc1} and part of Theorem \ref{TeorPrinc}.
\begin{Lem}\label{lem1}
Suppose that $A$ satisfies \eqref{A_conditions2} and $f\in L^{\infty}(B_{R}(x_0))$, where $B_R(x_0)$ is the open ball of radius $R>0$ centered at $x_0 \in \mathbb{R}^n$. Then, for $p>n$, there exists a family of radially symmetric supersolutions $v_{a,x_0}$ of \eqref{nonhomogeneousPr} in $B_{R}(x_0)\backslash\{x_{0}\}$
such that
\begin{equation}\label{lem1_est}
\bigg(\frac{1}{L}\bigg)^{\frac{1}{p-1}}a\frac{|x-x_0|^{\alpha}}{\alpha} \le v_{a,x_0}(x) \le  \bigg(\frac{1}{\delta}\bigg)^{\frac{1}{p-1}}\left(a+ \left( \frac{R^n \|f\|_{\infty}}{n}\right)^{\frac{1}{p-1}} \right)\frac{|x-x_0|^{\alpha}}{\alpha},
\end{equation}
for $a \ge \displaystyle  0$, where $\alpha = \frac{p-n}{p-1}$.
\label{lema1}
\end{Lem}
\begin{proof}
We start looking for radially symmetric solutions $v=v(r)$ of the equation
\begin{equation*}
-{\rm div}(|\nabla v |^{p-2} A(|\nabla v|) \nabla v) = \|f\|_{\infty} \, ,
\end{equation*}
where $r=|x-x_0| \in [0, R]$. This leads to the following ODE
$$ \frac{d}{dr}\left\{\,|v'|^{\,p-2}A\big(|v'|\big)v'\, \right\}+\frac{n-1}{r}\,|v'|^{\,p-2}A\big(|v'|\big)v' = -\|f\|_{\infty}\,.$$
Multiplying this equation by the integrating factor $r^{n-1}$ and performing an integration with respect to $r$, we get
\begin{equation}
\begin{split}
    |v'(r)|^{\,p-2}A\big(\,|v'(r)|\,\big)\,v'(r)\, r^{n-1}\,\, = C - \,\, \frac{\|f\|_{\infty}}{n}\,r^{n} \quad \text{for} \quad 0 < r \le R,
  \end{split} \label{radialSuperSolutionOfNonHomogeneous}
\end{equation}
where $C$ is some constant. Note that if f $v'(r) \ge 0$, this is equivalent to
\begin{equation}
\begin{split}
   v'(r)^{\,p-1}A(\,v'(r)\,)\, \, = \,\, \frac{1}{r^{n-1}}\left( C - \frac{\|f\|_{\infty}}{n}\,r^{n}\right).
  \end{split} \label{radialSuperSolutionOfNonHomogeneous2N}
\end{equation}
Let $\varphi(t):=t^{p-1}A(t)$ for $t \ge 0$. Condition \eqref{A_conditions2} implies that $\varphi$ has a nonnegative inverse $\varphi^{-1}$ defined in $[0,+\infty)$.
Moreover, from $ii)$ of \eqref{A_conditions2}, we can conclude that
\begin{equation}\label{phi}
 \Big(\,\frac{s}{L}\,\Big)^{\frac{1}{p-1}}\leq\varphi^{-1}(s)\leq \Big(\,\frac{s}{\delta}\,\Big)^{\frac{1}{p-1}}\,,\,\,\,\,\mbox{ for } s>0\,.
\end{equation}
Hence, for any $C\geq \frac{R^n \|f\|_{\infty}}{n} $, the function
\begin{equation}
 v(r) = \int_{0}^{r}\varphi^{-1}\left( \frac{1}{t^{n-1}}\left( C - \frac{\|f\|_{\infty}}{n}\,t^{n} \right) \right)dt\,,
\label{v-expression}
\end{equation}
is well defined and increasing in $[0,R]$.
Furthermore, it is a solution of \eqref{radialSuperSolutionOfNonHomogeneous2N} and, therefore, an increasing solution of \eqref{radialSuperSolutionOfNonHomogeneous}. Choosing
\begin{equation}
 C\,=\, \frac{R^n \|f\|_{\infty}}{n} +a^{\,p-1}\,,\,\,\,a\,\ge\, 0\,
\label{C_choice_0}
\end{equation}
we have
\begin{equation}
 v_a(r):= \int_{0}^{r}\varphi^{-1}\left( \,\frac{\|f\|_{\infty}}{n}\,(R^n-t^{n})\,t^{-n+1} + \,a^{\,p-1}t^{-n+1}\,\,\right)dt\,,\,\,\,a\geq 0\,.
\label{v-a-expression}
\end{equation}
Using \eqref{phi}, we have
\begin{equation*}
\begin{split}
    v_a(r) \,\geq\, &  \bigg(\,\frac{1}{L}\,\bigg)^{\frac{1}{p-1}}\int_{0}^{r}\,\left(\, \frac{\|f\|_{\infty}}{n}\,(R^n-t^{n}) + \,a^{\,p-1}\,\right)^{\frac{1}{p-1}}\,t^{-\frac{n-1}{p-1}}\,dt ,\\
  \end{split}
\end{equation*}
where noting that $t\leq R$, we obtain the lower bound
\begin{equation*}
    v_a(r) \,\geq \,  \left(\,\frac{1}{L}\,\right)^{\frac{1}{p-1}}\int_{0}^{r}\,a\, t^{-\frac{n-1}{p-1}}\,dt = \left(\,\frac{1}{L}\,\right)^{\frac{1}{p-1}}\,a\,\frac{r^{\alpha}}{\alpha}\,.
\end{equation*}
For the upper bound, using that $1/(p-1) < 1$, we can estimate
\begin{equation*}
\begin{split}
    v_a(r) \,\leq\, &  \left(\,\frac{1}{\delta}\,\right)^{\frac{1}{p-1}}\int_{0}^{r}\left( \frac{\|f\|_{\infty}}{n}(R^n-t^{n}) + a^{\,p-1} \right)^{\frac{1}{p-1}}\,t^{-\frac{n-1}{p-1}}\,dt \\
  \leq & \,\,  \left(\,\frac{1}{\delta}\,\right)^{\frac{1}{p-1}}\left( a+ \left(\frac{R^n \|f\|_{\infty}}{n}\right)^{\frac{1}{p-1}}\,\right) \,\frac{r^{\alpha}}{\alpha}\,.
  \end{split}
\end{equation*}
\end{proof}
\begin{Lem}\label{lemma2}
Suppose that $A$ satisfies \eqref{A_conditions2}, $p>n$, $f\in L^{\infty}(\mathbb{R}^n)$, and \eqref{f} holds. Then there exists a family of radially symmetric supersolutions $v_{a}$ of \eqref{nonhomogeneousPr} in $\mathbb{R}^n\backslash\{\,0\,\}$ satisfying
\\[5pt]$(a)$ $v_a(0)=0$ and $v_a(r)$ is nondecreasing in $(0,+\infty)$ for any $a\ge 0$;
\\$(b)$ $v_a$ is unbounded in $(0,+\infty)$ for $a > 0$; indeed, there exists a constant $c_0 =c_0(n,p,L)>0$ such that
 $$v_a(r)\, \ge\, c_0\,a \,r^{\alpha}  \quad {\rm for} \quad r \ge 0, \quad {\rm where} \quad \alpha=\frac{p-n}{p-1};$$
$(c)$ $v_0$ is bounded in $(0,+\infty)$; indeed, there exists a positive constant $C_0=C_0(n,p,\epsilon,C_f,\delta)$  such that
 $$v_0(r) \le C_0 $$
$(d)$ $v_a(r) \to v_0(r)$ as $a \to 0$ for any $r \in (0,+\infty)$.\\
\label{radialsupersol}
\label{upAndLowEstForVa-new1}
\end{Lem}
\begin{proof}
With no loss of generality, we can assume \eqref{f} holds for all $|x|\ge1$, with $|f| \le C_f$.
Hence, to obtain the desired supersolution we consider
\begin{equation}
g(r)\, =
\begin{cases}
\,\,\,\,\,\,C_f\,,\,&\,\,\text{for}\,\, r\le1\\
\,C_f\,r^{-p-\epsilon}\,,\,&\,\,\text{for}\,\, r\ge1
\end{cases}
\end{equation}
and look for radially symmetric solutions $v=v(r)$, $r=|x|$, of
\begin{equation}\label{eqf-new2}
-{\rm div}(|\nabla v |^{p-2} A(|\nabla v|) \nabla v)\, =\, g(r)\,,
\end{equation}
for $r>0$.  This leads to the ODE
\begin{equation} \label{eqf-new2-ODEversion}
    \frac{d}{dr}\left\{ \,|v'|^{\,p-2}A\big(|v'|\big)v'\,r^{n-1}\,\right\} = -g(r)\,r^{n-1}\,,
\end{equation}
and, therefore,
\begin{equation*}
\begin{split}
    |v'(r)|^{\,p-2}A\big(\,|v'(r)|\,\big)\,v'(r)\, r^{n-1}\,\, = \,\, -\int_{1}^{r}\,g(s)\,s^{n-1}\,ds \,+\,C
  \end{split}
\end{equation*}
where $C\,=\,|v'(1)|^{\,p-2}A\big(\,|v'(1)|\,\big)\,v'(1)$.
If $v'\geq0$, it follows that
\begin{equation*}
\begin{split}
   v'(r)^{\,p-1}A(\,v'(r)\,)\, r^{n-1}\,\, = \,\,-\int_{1}^{r}\,g(s)\,s^{n-1}\,ds \,+\,C\,.
  \end{split}
\end{equation*}
Note that $g(s)\,s^{n-1} \in L^1(0,+\infty)$. Then, for $C \ge \int_{1}^{+\infty}\,g(s)\,s^{n-1}\,ds$, we have
$$v'(r)= \varphi^{-1}\left( \,\frac{-\int_{1}^{r}\,g(s)\,s^{n-1}\,ds +C}{r^{n-1}}\,\,\right) \quad \text{for} \quad r > 0,$$
where $\varphi(t)=t^{p-1}A(t)$, since $\varphi^{-1}$ is defined in $[0,+\infty)$ as observed in the previous lemma. From this and \eqref{phi},
\begin{equation}
 v(r) = \int_{0}^{r}\varphi^{-1}\left( \,\frac{-\int_{1}^{\tau}\,g(s)\,s^{n-1}\,ds +C}{\tau^{n-1}}\,\,\right)d\tau\,\,
\label{v-a-expression-new2}
\end{equation}
is well defined and increasing for $r > 0$. It is also a solution of \eqref{eqf-new2-ODEversion}. Recalling the definition of $g$, we have for $0 \le r \le 1$
$$ v(r) \, = \,\int_{0}^{r}\varphi^{-1}\bigg( \,\frac{ \frac{C_f}{n}\,(\,1-\tau^n\,)+C}{\tau^{n-1}}\,\,\bigg)\,d\tau  $$
and, for $r\ge1$,
\begin{equation}
\begin{split}
 v(r) \, = &\,\int_{0}^{1}\varphi^{-1}\bigg( \,\frac{ \frac{C_f}{n}\,(\,1-\tau^n\,)+C}{\tau^{n-1}}\,\,\bigg)\,d\tau\,\,\,\,+\\
 &\,\,+\,\,\, \int_{1}^{r}\varphi^{-1}\bigg( \,\frac{\frac{C_f}{p-n+\epsilon}\,(\,\tau^{\,n-p-\epsilon}-1\,) +C}{\tau^{n-1}}\,\,\bigg)\,d\tau \, .
 \end{split}
\label{v-a-expression2}
\end{equation}
We prove this lemma only for $r \ge 1$, since for $0 \le r \le 1$ the argument is simpler.
Then, choosing
\begin{equation}\label{C_choice}
C = \frac{C_f}{p-n+\epsilon} + a^{\,p-1} \,,\,\,\,a\,\ge\, 0\,,
\end{equation}
it follows that
\begin{equation}
\begin{split}
 v_a(r) \,= &\,\int_{0}^{1}\varphi^{-1}\bigg( \,\frac{ \frac{C_f}{n}\,(\,1-\tau^n\,)+\frac{C_f}{p-n+\epsilon} + a^{\,p-1}}{\tau^{n-1}}\,\,\bigg)\,d\tau\,\,\,\,+\\
 &\,\,+\,\,\, \int_{1}^{r}\varphi^{-1}\bigg( \frac{\frac{C_f}{p-n+\epsilon}\,\tau^{\,n-p-\epsilon} + a^{\,p-1}}{\tau^{n-1}} \,\,\,\bigg)\,d\tau\\
 \end{split}
\label{v-a-expression3}
\end{equation}
is a family of increasing solutions of \eqref{eqf-new2-ODEversion}. Using \eqref{phi} we can estimate $v_a$ from below:
\begin{equation*}
\begin{split}
    v_a(r) \,\geq\, & \,\bigg(\frac{1}{L}\bigg)^{\frac{1}{p-1}}\int_{0}^{1}\bigg(\,\frac{C_f}{n}\,(\,1-\tau^n\,)+\frac{C_f}{p-n+\epsilon}\,+ a^{\,p-1}\,\bigg)^{\frac{1}{p-1}}\,\tau^{\,-\frac{n-1}{p-1}}\,d\tau\\[5pt]
 &+ \bigg(\frac{1}{L}\bigg)^{\frac{1}{p-1}}\int_{1}^{r}\bigg(\,\frac{C_f}{(p-n+\epsilon)} \tau^{\,n-p-\epsilon} +a^{\,p-1} \,\bigg)^{\frac{1}{p-1}}\,\tau^{\,-\frac{n-1}{p-1}}\,d\tau\,\\[5pt]
 \geq\, & \,\bigg(\frac{1}{L}\bigg)^{\frac{1}{p-1}}\bigg(\, \int_{0}^{1} a \,\tau^{\,-\frac{n-1}{p-1}}\,d\tau\,+\,\int_{1}^{r} a\,\tau^{\,-\frac{n-1}{p-1}}\,d\tau\,\bigg) \geq \bigg(\frac{1}{L}\bigg)^{\frac{1}{p-1}}\,a\,\frac{r^{\,\alpha}}{\alpha}\, .
\end{split}
\end{equation*}
For the upper bound, we can estimate from \eqref{v-a-expression3} and \eqref{phi},
\begin{equation}\label{upper_bound}
\begin{split}
 v_0(r) \, = &\,\int_{0}^{1}\varphi^{-1}\bigg( \,\frac{ \frac{C_f}{n}\,(\,1-\tau^n\,)+\frac{C_f}{p-n+\epsilon}}{\tau^{n-1}}\,\,\bigg)\,d\tau\\[5pt]
 &\,\,+\,\,\, \int_{1}^{r}\varphi^{-1}\bigg( \,\frac{C_f}{p-n+\epsilon}\,\frac{\,\tau^{\,n-p-\epsilon} }{\tau^{n-1}}\,\,\bigg)\,d\tau\\[5pt]
  \le &\,\bigg(\frac{1}{\delta}\bigg)^{\,\frac{1}{p-1}}\,\int_{0}^{1}\,\bigg( \, \frac{C_f}{n}\,(\,1-\tau^n\,)+\frac{C_f}{p-n+\epsilon}\,\,\bigg)^{\,\frac{1}{p-1}}\,\tau^{-\frac{n-1}{p-1}}\,d\tau \\[5pt]
 &\,\,+\,\,\, \bigg(\frac{1}{\delta}\bigg)^{\,\frac{1}{p-1}}\,\int_{1}^{r}\bigg( \,\frac{C_f}{p-n+\epsilon}\,\,\bigg)^{\,\frac{1}{p-1}}\,\tau^{\,\frac{-p-\epsilon+1}{p-1}}\,d\tau\\[5pt]
 \le &\,\bigg(\frac{1}{\delta}\bigg)^{\,\frac{1}{p-1}}\,\int_{0}^{1}\,\bigg( \, \frac{C_f (p+\epsilon )}{n(p-n+\epsilon)} \bigg)^{\,\frac{1}{p-1}}\,\tau^{-\frac{n-1}{p-1}}\,d\tau \\[5pt]
 &\,\,+\,\,\, \bigg(\frac{1}{\delta}\bigg)^{\,\frac{1}{p-1}}\,\int_{1}^{r}\bigg( \,\frac{C_f}{p-n+\epsilon}\,\,\bigg)^{\,\frac{1}{p-1}}\,\tau^{\,\frac{-p-\epsilon+1}{p-1}}\,d\tau\\[5pt]
     \leq\,&\,\bigg(\frac{C_f\,(\,p+\epsilon\,)}{\delta \, n (\,p-n+\epsilon\,)}\bigg)^{\frac{1}{p-1}}\,\Big(\,\frac{1}{\alpha}+\frac{p-1}{\epsilon}\,\Big).
  \end{split}
\end{equation}
The item $(d)$ is a direct consequence of \eqref{v-a-expression3} and the continuity of $\varphi^{-1}$.
\end{proof}

\subsection{Harnack inequality}

In the proof of Theorem 2, beyond the use of barriers, we have to apply a Harnack inequality and a symmetrization result. The following theorem about Harnack inequality is established in Theorems 5, 6 and 9 of \cite{S0}:

\begin{Teo}\label{HarnakIneq}
 Let $u$ be a nonnegative weak solution of \eqref{nonhomogeneousPr} on an open ball $B_R$. For $p\le n$, assume that $f\in L^{\frac{n}{p-\theta}}(B_R)$ for some $\theta\in(0,1)$. For $p> n$, assume that $f\in L^{1}(B_R)$.  Then, for any $\sigma\in (0,1)$,
\begin{equation}
\underset{B_{\,\sigma R}}{\,\sup\,}\,u\,\le\,C\,\Big(\,\underset{B_{\,\sigma R}}{\,\inf\,}\,u + K(R)\,\Big)
\end{equation}
where $C$ depends on $n,p,\sigma,\delta,L$ and, in the case $p\le n$, also on $\theta$. Moreover
\begin{equation}\label{HarnacK}
K(R)\,=\,\Big(\,R^{\,\theta}\,\|\,f\,\|_{L^{\,\frac{n}{p-\theta}}(B_R)}\,\Big)^{\,\frac{1}{p-1}}
\end{equation}
if $p\le n$, and
\begin{equation}\label{HarnacK2}
K(R)\,=\,\Big(\,R^{\,p-n}\,\|\,f\,\|_{L^{\,1}(B_R)}\,\Big)^{\,\frac{1}{p-1}}
\end{equation}
if $p>n$.
\end{Teo}

The result above can be extended with no difficulty to arbitrary compact subsets. We can extract the corollaries below which give the Harnack inequality for solutions on exterior domains over the spheres $S_R$, for all $R$ large, with $C$ independent of $R$.
\begin{Cor}\label{SRHarnack1}
Let $u$ be a non-negative weak solution of \eqref{nonhomogeneousPr} on $\mathbb{R}^n\setminus\overline{B_1}$ and assume $f$ satisfy condition $\eqref{f}$. Then, for all $R\ge4$,
\begin{equation}\label{harnackSR1}
\underset{S_{\,R}}{\,\sup\,}\,u\,\le\,C\,\Big(\,\underset{S_{R}}{\,\inf\,}\,u + R^{\,-\frac{\epsilon}{p-1}}\,\Big)
\end{equation}
where $C$ depends only on $n,p,\delta,L$.
\begin{proof}
We can cover $S_{R}$ with $N$ balls $B_i=B_{R/2}(x_i)$ whose centers $x_i$ lie on $S_{R}$ and $N$ does not depend on $R$. Ordering these balls so that $B_i\cap B_{i+1}\neq\varnothing$, we have
\begin{equation}\label{supinf1}
\underset{B_{i}}{\,\inf\,}\,u\,\le\,\underset{B_{i+1}}{\,\sup\,}\,u\,.
\end{equation}
Now we apply the previous theorem on each ball $B_{3R/4}(x_i)\subset\mathbb{R}^n\setminus \overline{B_1}$, with $\sigma=2/3$. Using $\eqref{f}$, a computation of the norms of $f$ shows that, for any case, $K$ can be estimated as
$$ K(3R/4)\,\le\,C\,R^{\,-\frac{\epsilon}{p-1}}$$
for some constant $C$ depending only on $n,p$, so we have by the theorem
\begin{equation}\label{harnackapplied1}
\underset{B_i}{\,\sup\,}\,u\,\le\,C\,\Big(\,\underset{B_i}{\,\inf\,}\,u + R^{\,-\frac{\epsilon}{p-1}}\,\Big)
\end{equation}
where $C$ depends only on $n,p, L$ and, in case $p\le n$, of a chosen $\theta\in(0,1)$.
Then by combining inequalities \eqref{supinf1} and \eqref{harnackapplied1} it follows, for all $i,j\in\{\,1,\ldots,N\,\}$,
$$
\underset{B_i}{\,\sup\,}\,u\,\le\,C\,\Big(\,\underset{B_j}{\,\inf\,}\,u + R^{\,-\frac{\epsilon}{p-1}}\,\Big)
$$
after a proper redefinition of $C$ depending only on $N$. This leads to \eqref{harnackSR}.
\end{proof}

\end{Cor}
\begin{Cor}\label{SRHarnack}
Let $u$ be a non-negative weak solution of \eqref{nonhomogeneousPr} on $\mathbb{R}^n\backslash B_1$. Assume that $p\le n$ and that there is some $\theta\in(0,1)$ such that $f\in L^{\frac{n}{p-\theta}}$. Then, for all $R$ sufficiently large,
\begin{equation}\label{harnackSR}
\underset{S_{\,R}}{\,\sup\,}u \le C\,\Big(\,\underset{S_{R}}{\,\inf\,}u + K(R)\Big),
\end{equation}
where
\begin{equation}\label{KR}
K(R)\,=\,\Big(\big(R/4\big)^{\,\theta}\,\|\,f\,\|_{L^{\frac{n}{p-\theta}}(\,\mathbb{R}^n\backslash B_{R/4}(0)\,)}\Big)^{\,\frac{1}{p-1}}
\end{equation}
and $C$ depends on $n,p,\theta,\delta, \Gamma$.
\end{Cor}

\begin{proof}[Proof]
Defining the balls $B_i$ in the same way as we did in the previous corollary, we have
\begin{equation}\label{supinf}
\underset{B_{i}}{\,\inf\,}\,u\,\le\,\underset{B_{i+1}}{\,\sup\,}\,u.
\end{equation}
Now we apply Theorem \ref{HarnakIneq} for each ball $B_{3R/4}(x_i)\subset\mathbb{R}^n\backslash B_1$, with $\sigma=2/3$. We obtain
\begin{equation}\label{harnackapplied}
\underset{B_i}{\,\sup\,}u \le C\,\Big(\underset{B_i}{\,\inf\,}\,u + K(R)\Big)
\end{equation}
with
\begin{equation*}
\begin{split}
K(R)\,\,=&\,\,\Big(\big(3R/4\big)^{\,\theta}\,\|\,f\,\|_{L^{\frac{n}{p-\theta}}(\,B_{3R/4}(x_i)\,)}\Big)^{\,\frac{1}{p-1}}\\
\le&\,\,3^{\,\frac{1}{p-1}}\,\Big(\big(R/4\big)^{\,\theta}\,\|\,f\,\|_{L^{\frac{n}{p-\theta}}(\,\mathbb{R}^n\backslash B_{R/4}(0)\,)}\Big)^{\,\frac{1}{p-1}}.
\end{split}
\end{equation*}
Then combining inequalities \eqref{supinf} and \eqref{harnackapplied} yields, for all $i,j\in\{1,\ldots,N\}$,
$$
\underset{B_i}{\,\sup\,}u\,\le\,C\,\Big(\underset{B_j}{\,\inf\,}\,u + K(R)\Big)
$$
after a proper redefinition of $C$ depending only on $N$. This leads to \eqref{harnackSR}, as it is clear we can choose $K$ above as in \eqref{KR}, redefining $C$ if necessary.
\end{proof}

\subsection{Schwarz symmetrization}

Now we recall some definitions and useful results about symmetrization. For an exhaustive treatment about this topics we refer to Hardy, Littlewood and P\'olya \cite{HLP}, Talenti \cite{T1}, Alvino, Lions and Trombetti \cite{ALT} and Brothers and Ziemer \cite{BZ}.

First, if $\Omega$ is an open bounded set in
$\mathbb{R}^n$ and $u: \Omega \to \mathbb{R}$ is a measurable
function, the distribution function of $u$ is given by
$$\mu_{u}(t)=|\{ x \in \Omega : |u(x)| > t \}|, \quad {\rm for } \; t \ge 0.$$
The decreasing rearrangement of $u$, also called the {\it generalized
inverse} of $\mu_u$, is defined by
$$u^{*}(s)= \sup \{ t \ge 0 : \mu_u (t) \ge s \}. $$
If $\Omega^{\sharp}$ is the open ball in $\mathbb{R}^n$, centered at
$0$, with the same measure as $\Omega$ and $\omega_n$ is the
measure of the unit ball in $\mathbb{R}^n$, the function
$$ u^{\sharp}(x) = u^*( \omega_n |x|^n), \quad {\rm for} \quad x \in \Omega^{\sharp} $$
is the spherically symmetric decreasing rearrangement of $u$. It is also
called the Schwarz symmetrization of $u$.

\begin{remark}{\rm
\label{obs2} Let $v,w$ be integrable functions in $\Omega$ and let
$g:\mathbb{R}\to\mathbb{R}$ be a non-decreasing non-negative function.
Then
$$ \int_{\Omega} g(|v(x)|) \; dx = \int_0^{|\Omega|} g(v^*(s)) \; ds
= \int_{\Omega^{\sharp}} g(v^{\sharp}(x)) \; dx .$$ Hence, if
$\mu_v(t) \ge \mu_w(t) $ for all $t>t_1>0$, it follows that
$$ \int_{t_1<v} g(v(x)) \; dx  =  \! \int_0^{\mu_u(t_1)} g( v^*(s)) \; ds \ge
\int_0^{\mu_w(t_1)} g( w^*(s)) \; ds = \! \int_{t_1<w} g(w(x)) \;
dx,
$$ since $v^*(s) \ge w^*(s)$
for $s \le \mu_w(t_1) $.
Finally, the P\'olya-Szeg\"o principle (see, for instance, Brothers and Ziemer \cite{BZ}) states that
$$\int_{\Omega} | Dv(x) |^2 \; dx  \; \ge \; \int_{\Omega^{\sharp}} | Dv^{\sharp} (x) |^2 \;
dx, \quad {\rm for} \; \; v \in H_0^1 (\Omega). $$
This inequality also holds if we replace $\Omega$ and $\Omega^{\sharp}$ by $\{t_1 < v < t_2 \}$
and $\{t_1 < v^{\sharp} < t_2 \}$, respectively.}
\end{remark}

The next theorem taken from Talenti's work \cite{Talenti} is a result that compares the solutions of some Dirichlet problems with the solutions of their symmetrized versions.  As references for works in this line of research, we can cite also Talenti \cite{T1}, Trombetti and Vasquez \cite{TV}, Kesavan \cite{kesavan}, and Bonorino and Montenegro \cite{BM}, among others.

\begin{Teo}\label{teorematalenti}
Consider a weak solution $u$ of \eqref{main} on a bounded domain $\Omega$, where $A$ satisfies \eqref{A_conditions2}, and assume $f\in L^1$. Then, $u^{\sharp}$, the spherically symmetric rearrangement of $u$, satisfies
\begin{equation}\label{Talenti}
u^{\sharp}(x)\,\le\,\underset{ \partial \Omega}{\,\sup\,}|u| \,+\,\int_{\omega_n|x|^{n}}^{|\Omega|}\, \varphi^{-1}\bigg(\frac{r^{-1+1/n}}{n\,\omega_n^{1/n}}\int_{0}^{r}f^{*}(s)\,ds\bigg)\,\frac{r^{-1+1/n}}{n\,\omega_n^{1/n}}\,dr,
\end{equation}
where $\varphi(t)=t^{\,p-1}A(t)$.
\end{Teo}
\begin{proof}
\cite[Theorem 1]{Talenti}.
\end{proof}
From the Theorem above, we can derive the following statement.
\begin{Cor}\label{cortalenti}
Under the same hypotheses from Theorem \ref{teorematalenti}, $u^\sharp$ satisfies
\begin{equation}\label{Talenti0}
\underset{ \Omega^{\sharp}}{\,\sup\,}u^{\sharp}\le\,\underset{ \partial \Omega}{\,\sup\,}|u| \,+\Big(\frac{1}{n\,\omega_n\,\delta}\Big)^{\frac{1}{p-1}}\,\int_{0}^{\big(\frac{|\Omega|}{\omega_n}\big)^{1/n}} \rho^{-\frac{n-1}{p-1}}\bigg(\int_{B_\rho}f^{\sharp}(x)\,dx\bigg)^{\frac{1}{p-1}}\,d\rho.\\
\end{equation}
\end{Cor}

\begin{proof}
Making the changes of variables
\begin{equation*}
s=\omega_n\,t^n\,,\,\,\,ds=n\,\omega_n\,t^{n-1}dt\,
\end{equation*}
in \eqref{Talenti}, it follows
\begin{equation*}
u^{\sharp}(x)\,\le\,\underset{ \partial \Omega}{\,\sup\,}|u| \,+\,\int_{\omega_n |x|^n}^{|\Omega|} \varphi^{-1}\bigg( \frac{r^{-1 + \frac{1}{n}}}{\omega_n^{-1+\frac{1}{n}}} \, \int_{0}^{\big( \frac{r}{\omega_n} \big)^{\frac{1}{n}}}f^{*}(\omega_n\,t^n) \, t^{n-1}\,dt\bigg)\, \frac{r^{-1 + \frac{1}{n}}}{n \omega_n^{\frac{1}{n}}}  \,dr.
\end{equation*}
If we consider now the change of variables
\begin{equation*}
r=\omega_n\,\rho^n\,,\,\,\,dr=n\,\omega_n\,\rho^{n-1}d\rho,
\end{equation*}
with
\begin{equation*}
r^{-1+1/n}\,=\,\omega_n^{-1+1/n}\,\rho^{-n+1},
\end{equation*}
it follows
\begin{equation*}
u^{\sharp}(x)\,\le\,\underset{ \partial \Omega}{\,\sup\,}|u| \,+\,\int_{|x|}^{\big(\frac{|\Omega|}{\omega_n}\big)^{1/n}} \varphi^{-1}\bigg(\rho^{-n+1}\int_{0}^{\rho}f^{*}(\omega_n\,t^n)\,t^{n-1}\,dt\bigg)\,d\rho.
\end{equation*}
Note that, from Remark \ref{obs2}, we get
$$ \int_{0}^{\rho}f^{*}(\omega_n\,t^n)\,t^{n-1}\,dt\,=\,\frac{1}{n\omega_n}\,\int_{B_\rho}f^{\sharp}(x)\,dx,$$
and so,
\begin{equation*}
u^{\sharp}(x)\,\le\,\underset{ \partial \Omega}{\,\sup\,}|u| \,+\,\int_{|x|}^{\big(\frac{|\Omega|}{\omega_n}\big)^{1/n}} \varphi^{-1}\bigg(\frac{\rho^{-n+1}}{n\omega_n}\,\int_{B_\rho}f^{\sharp}(x)\,dx\bigg)\,d\rho.
\end{equation*}
Using the lower estimate in \eqref{phi}, we get
\begin{equation*}
u^{\sharp}(x)\,\le\,\underset{ \partial \Omega}{\,\sup\,}|u| \,+\Big(\frac{1}{n\,\omega_n\,\delta}\Big)^{\frac{1}{p-1}}\,\int_{|x|}^{\big(\frac{|\Omega|}{\omega_n}\big)^{1/n}} \rho^{-\frac{n-1}{p-1}}\,\bigg( \int_{B_\rho}f^{\sharp}(x)\,dx\bigg)^{\frac{1}{p-1}}\,d\rho,
\end{equation*}
from which we obtain \eqref{Talenti0}.
\end{proof}

\section{Proof of Theorem 1}

\hspace{1,2cm} The uniqueness of solutions is a direct consequence of the comparison principle in
\cite[Theorem 2]{BSZ}, presented in Preliminaries. For the existence, we split the proof into three steps.\\
\par \textit{1. Construction of a bounded solution.}\\

\par We consider a decreasing sequence of smooth compact sets $K_m$ satisfying, for all $m$\\
\hspace*{0,5in}$i)$ $K\Subset K_{m+1}\Subset K_m$\\
\hspace*{0,5in}$ii)$ $dist(\partial K, \partial K_m)\,\rightarrow\,0\,.$\\
\par Taking an increasing sequence of radii $R_m\,\rightarrow\,+\infty$, with $K_m\Subset B_{R_1}$, for all $m$, we continuously extend $\phi$ to the whole $\mathbb{R}^n$, keeping fixed $\sup\,|\phi|$ and setting $\phi = 0$ in $\mathbb{R}^n\setminus B_{R_1}$. We then look for the domains
$\Omega_m := B_{R_m}\setminus K_m$ and the problems
 \begin{equation}\label{problems}
 \begin{cases}
    \;\;\; -{\rm div}\big(\,|\,\nabla u\,|^{p-2}A(\,|\,\nabla u\,|\,)\nabla u\,\big)=f\,\, &\text{ in}\,\,\Omega_m\, \\
    \;\;\;\;\;\; u = \phi &  \text{ in } \partial K_m\,\\
    \;\;\;\;\;\; u = 0 &  \text{ in } \mathbb{R}^n \setminus B_{R_m}\,.
    \end{cases}
\end{equation}
By Theorem \ref{existenceForContinuousBoundaryData}, each of those problems has a weak solution
$u_m\in C(\overline{\Omega}_m) \cap C^1(\Omega_m)$.
\par  Now let $v_0$ be the supersolution given by Lemma \ref{lemma2} and assume, with no loss of generality, that $K$ contains the origin $0\in\mathbb{R}^n$, so that $0\notin\Omega_m$, for all $m$.
Hence, the function $v_0+\sup\phi$ is then a supersolution in $\Omega_m$, with $u_m \leq v_0+\sup\phi$ on $\partial \Omega_m$, for all $m$. Since $v_0+\sup\phi\leq C_0+\sup\phi$, we obtain by the comparison principle the uniform bound
\begin{equation}\label{unifbound}
\sup\,u_m\,\leq\, C_0+\sup\,\phi\,,\,\,\,\text{ for all } m\,.
\end{equation}
 Moreover, by the local H\"{o}lder regularity result in \cite[Theorem 1.1, p. 251]{Ladyzhenskaya}, there exists a $\gamma>0$ such that, for each compact $V$ of $\mathbb{R}^n\setminus K$, there is a constant $C>0$, depending on the parameters of the equation and $V$, such that
 $$|\,u_m(x)-u_m(y)\,|\,\le\,C|\,x-y\,|^\gamma\,,\,\,\,\text{for all }\,x,y\,\in\,V\,.$$
Thus, $u_m$ is also equicontinuous on $V$ and, by Arzel\'a-Ascoli's Theorem, we can obtain a subsequence of $u_m$ converging uniformly on $V$ to some continuous function. Considering then a sequence of compact sets $V_k$ such that $B\setminus K = \bigcup V_k$, by a standard diagonal argument, we can find a continuous function $u$ on $\mathbb{R}^n\setminus K$ and a subsequence of $u_m$ that converges to $u$ uniformly on any compact of $\mathbb{R}^n\setminus  K$. Furthermore, from Theorem 1 of \cite{Tolksdorf}, $\nabla u_m$ is uniformly bounded and equicontinuous in compacts. Hence, using again Arzel\'a-Ascoli's Theorem and a diagonal argument, it follows that up to a subsequence $\nabla u_m \to \nabla u$ uniformly in compacts, from which we can conclude that $u$ is a weak solution of \eqref{main} in $\mathbb{R}^n\setminus K$.
\hspace*{\fill}$\square$\\
\par \textit{2. Continuity of $u$ on the boundary.}\\
\par Let $x_0\in\partial K$, $\epsilon>0$ and consider by Lemma 1 the supersolutions $\,v_{a,x_0}\,$ on a large ball $B(x_0)\setminus\{x_0\}$. By the continuity of $\phi$, there is some $R>0$ such that
\begin{equation*}
|\,\phi(x)-\phi(x_0)\,|\,<\,\epsilon\,,\,\,\,\text{ for }|x-x_0|<R
\end{equation*}
so that
\begin{equation*}
\phi(x_0)+v_{a,x_0}(x)+\epsilon\,\geq\,\phi(x)\,,\,\,\,\text{ for }|x-x_0|\,<\,R\,,\,\,\,a\geq0\,.
\end{equation*}
We then choose $a$ sufficiently large in \eqref{lem1_est} to make
\begin{equation*}
\phi(x_0)+v_{a,x_0}(x)+\epsilon\,\ge\,\sup\,\phi\,,\,\,\,\text{ for }|x-x_0|\,\ge\,R\,.
\end{equation*}
Therefore, the function
\begin{equation*}
w_{a,x_0}^{+}\,:=\,\phi(x_0)+v_{a,x_0}+\epsilon
\end{equation*}
satisfies $w_{a,x_0}^{+}\,\ge\,\phi$, so that, in particular,
\begin{equation*}
w_{a,x_0}^{+}\,\ge\,\phi\,=u_m\,\,\,\,\,\,\,\,\text{in}\,\,\,\partial K_m\,,\,\,\,\,\text{for all}\,\,m\,.
\end{equation*}
By taking $a$ larger if necessary, we can also make
\begin{equation*}
w_{a,x_0}^{+}\,\ge\,u_m\,\,\,\,\,\,\,\,\,\text{in}\,\,\,\partial B\,,\,\,\,\,\text{for all}\,\,m\,.
\end{equation*}
Then by applying the comparison principle on $B\setminus K_m$ we obtain
\begin{equation*}
w_{a,x_0}^{+}\,\ge\,u_m\,\,\,\,\,\,\,\,\text{in}\,\,\,B\setminus K_m\,,\,\,\,\,\text{for all}\,\,m\,
\end{equation*}
from which follows
\begin{equation*}\label{barrier_ineq2}
w_{a,x_0}^{+}\,\ge\,u\,\,\,\,\,\,\,\,\text{in}\,B\setminus K
\end{equation*}
since $u_m$ converges to $u$ on $B\setminus K$. Finally, this implies
\begin{equation*}\label{limsup_ineq}
\underset{x\rightarrow x_0}{\limsup }\,u(x)\, \le \, \underset{x\rightarrow x_0}{\limsup }\,w_{a,x_0}^{+}(x)\, = \, \phi(x_0) + \epsilon\,
\end{equation*}
and by arbitrariness of $\epsilon$ we conclude
\begin{equation*}\label{limsup}
\underset{x\rightarrow x_0}{\limsup }\,u(x)\, \le \, \phi(x_0)\,.
\end{equation*}
By an analogous argument with the subsolution $w_{a,x_0}^{-}:=\phi(x_0)-v_{a,x_0}-\epsilon$ we can obtain the lower bound
\begin{equation*}\label{liminf}
\underset{x\rightarrow x_0}{\liminf }\,u(x)\, \ge \, \phi(x_0)\,
\end{equation*}
concluding the result.\\
\hspace*{\fill}$\square$\\

\par  \textit{3. Global H\"{o}lder Continuity of $u$.}\\
\par Assume $\phi $ is $\alpha$-H\"{o}lder continuous in $K$, with $\alpha = \frac{p-n}{p-1}$. We will show $u$ is $\alpha$-H\"{o}lder continuous in $\mathbb{R}^n$.\\
\par  Let $y\in K$, $R>0$ and $v_a=v_{a,y}$ a supersolution in $B_{R}(y)\setminus\{y\}$ as given in Lemma \ref{lem1}.
We claim that for all $a$ sufficiently large
$$  \phi(y) - v_a\, \leq \,u \,\leq \,\phi(y) + v_a\;\mbox{ in }B_{R}(y)$$
for all $y\in K$. For this, putting $C=|\phi|_{\alpha}$, the H\"{o}der seminorm of $\phi$ in $K$, we have by definition
$|\,\phi(z)-\phi(y)\,|\leq C|\,z-y\,|^{\alpha}$, for all $z\in K$, hence
\begin{equation*}
 \phi(y)-C|\,z-y\,|^{\alpha}\,\leq\,\phi(z)\,\leq\, \phi(y)+C|\,z-y\,|^{\alpha}\,\,\,\mbox{ for all }z\in K\,.
\end{equation*}
Now by estimate \ref{lem1_est} we see that for all $a$ large enough $v_a$ satisfies
$$ \,C\,|\,x-y\,|^{\alpha}\, \le\, v_a(x) \,\mbox{ for all }\,x\in B_{R}(y)$$
so that from last inequality it follows
\begin{equation}\label{eq4'}
  \phi(y) - v_a\,\leq\,\phi\, \leq \, \phi(y) + v_a\;\mbox{ in } K\cap B_{R}(y)\,.
\end{equation}
Now taking $a$ larger if necessary, by estimate \eqref{lem1_est} we can also ensure that
$$ |\,u-\phi(y)\,|\,\le\,2\sup |u|\, \le\, v_a \,\mbox{ in }\, \partial B_{R}(y)$$
and so
\begin{equation*}
 \phi(y) - v_a\,\le\,u\, \le\, \phi(y) + v_a \,\mbox{ in }\partial B_{R}(y)\,.
\end{equation*}
Hence, noting that $\phi=u$ in $K$ in \eqref{eq4'}, we see the inequality above holds on $ \partial (B_{R}(y)\setminus K)\,$, so that, by the comparison principle, it extends to $B_{R}(y)\setminus  K$, concluding the claim. Notice the parameter $a$ depends only on $|\phi|_\alpha$ and $\sup u$.\\

\par Now let $x_0\in\mathbb{R}^{n}\setminus K$. It is enough to prove H\"{o}lder continuity on a neighbourhood of $K$ so we may assume $d(x_0,K)<R$. By the claim we have, in particular for all $y\in B_R(x_0)\cap K,$
$$  \phi(y)-v_{a,y}(x_0)\,\leq u(x_{0})\,\leq \,\phi(y)+v_{a,y}(x_0)$$
These inequalities give
$$ u(x_{0})- v_{a,y}(x_0)\,\leq\, \phi(y)\, \leq\, u(x_{0})+v_{a,y}(x_0)$$
and, as $\phi=u$ in $K$, we get
\begin{equation}\label{eq05}
u(x_{0})- v_{a,y}(x_0)\,\leq \,u(y) \,\leq \,u(x_{0})+ v_{a,y}(x_0)
\end{equation}
for all $ y\in\, B_R(x_0)\cap K$. Using the upper estimate \eqref{lem1_est} we have for some constant $C_1$
$$ v_{a,y}(x)\, \le\,C_1\,|\,x-y\,|^{\alpha}\,  \,\mbox{ for all }\,x\in B_{R}(y)$$
and, in particular,
$$ v_{a,y}(x_0)\, \le\,C_1\,|\,x_0-y\,|^{\alpha}\,.$$
Now using the lower estimate in \eqref{lem1_est} for the supersolution $v_{a,x_0}$ centered at $x_0$ we can obtain
$$ \,C\,|\,x_0-y\,|^{\alpha}\, \le\, v_{a,x_0}(y) \,$$
and so
$$ v_{a,y}(x_0)\, \le\,\frac{C_1}{C}\,v_{a,x_0}(y)\,.$$
From \eqref{eq05} it follows
\begin{equation}\label{eq06}
 u(x_{0})- \frac{C_1}{C}\,v_{a,x_0}(y)\,\leq \,u(y) \,\leq \,u(x_{0})+ \frac{C_1}{C}\,v_{a,x_0}(y)\,
\end{equation}
for all $y\in B_{R}(y)\cap K$. Provided that $C_1/C>1$, we have that $\frac{C_1}{C}\,v_{a,x_0}$ is also a supersolution in $B_R(x_0)$ and by the previous choice of $a$, still  $\frac{C_1}{C}\,v_{a,x_0}\ge 2\sup|u|$ in $\partial B_R(x_0)$. Therefore, \eqref{eq06} holds for all $y\in \partial (B_R(x_0)\setminus K)$ and by the comparison principle it also holds on $B_R(x_0)\setminus K$, so we have
\begin{equation*}
 u(x_{0})- \frac{C_1}{C}\,v_{a,x_0}(x)\,\leq \,u(x) \,\leq \,u(x_{0})+ \frac{C_1}{C}\,v_{a,x_0}(x)\,
\end{equation*}
for all $x\in B_R(x_0)$. Using again the upper estimate in \eqref{lem1_est} for $v_{a,x_0}$ we get
$$ v_{a,x_0}(x)\, \le\,C_1\,|\,x-x_0\,|^{\alpha}\,  \,\mbox{ for all }\,x\in B_{R}(x_0)$$
which gives
\begin{equation*}
 u(x_{0})- \frac{C_1^2}{C}\,|\,x-x_0\,|^{\alpha}\,\leq \,u(x) \,\leq \,u(x_{0})+ \frac{C_1^2}{C}\,|\,x-x_0\,|^{\alpha}\,
\end{equation*}
for all $x\in B_{R}(x_0)$, which is the H\"{o}lder continuity of $u$ at $x_0$. This concludes the statement as $x_0$ is arbitrary and the H\"{o}lder seminorm of $u$ is then bounded by $C_1^2/C$, independently of $x_0$.\\
\hspace*{\fill}$\square$\\

\begin{remark}
If we replace $\mathbb{R}^n$ by a bounded domain $\Omega \supset K$ in this theorem, using the same argument we can guarantee the existence of a bounded weak solution $u \in C(\overline{\Omega\setminus K})\cap C^1(\Omega\setminus K)$ of \eqref{nonhomogeneousPr}, under the same hypotheses on $A$, $K$, $p$, and $\phi$. With respect to $f$, we only have to assume that $f$ is bounded, condition \eqref{f} is not necessary. Indeed, since $\Omega$ is bounded, there is no need of global barriers given by Lemma \ref{lemma2}. The local barriers of Lemma \ref{lem1} are sufficient. For uniqueness, we would have to add, for instance, a boundary condition $u = \tilde{\phi}$ on $\partial \Omega$, otherwise there are infinitely many solutions.
\end{remark}

\section{Limit at infinity for $f$ bounded in $L^p$ spaces}

\par In this section, we prove $(b)$ of Theorem \ref{TeorPrinc}, that is, we suppose that $1<p<n$ and $f$ satisfies \eqref{Lr}, \eqref{Ltheta} and \eqref{Kgoes0}. One of the main tools is Corollary \ref{cortalenti}, which is a direct application of the comparison results established by Talenti in \cite{T1,Talenti}. We also use a Harnack inequality as stated in Corollary \ref{SRHarnack}.

\

\par \textit{1. Proof of item (b) of Theorem \ref{TeorPrinc}:} \\

 Let $u$ be a bounded weak solution of $$-{\rm div}\big(|\nabla u|^{p-2}A(|\nabla u|)\nabla u\big)\,=\,f $$ on $\mathbb{R}^n\backslash B_1$ and set $ m\,=\,\underset{|x|\rightarrow\infty}{\liminf}\,u$.
For a given $\varepsilon >0$ let $R_0>0$ be such that
\begin{equation*}
u(x)\,>\,m-\varepsilon\,\,\,\text{for all}\,\,x\,\,\text{such that}\,\,|x|\ge R_0\,,
\end{equation*}
so that the function
\begin{equation*}
v\,=\,u-m+\varepsilon\,
\end{equation*}
is a positive solution on $\mathbb{R}^n\backslash B_{R_0}$. We pick up a sequence $(x_k)$ with $|x_k|\rightarrow \infty$ and $R_0<|x_k|<|x_{k+1}|$ such that
\begin{equation*}
u(x_k)\,\le\,m+\epsilon,
\end{equation*}
hence
\begin{equation}\label{2epsilon}
v(x_k)\,\le\,2\epsilon.
\end{equation}
Now let $R_k = |x_k|$, $S_{R_k} = \partial B_{R_k}(0)$. By applying Corollary \ref{SRHarnack} to $v$ we get
\begin{equation*}
\underset{S_{\,R_k}}{\,\sup\,v}\,\le\,C\,\Big(\,\underset{S_{\,R_k}}{\,\inf\,}\,v + K(R_k)\,\Big)
\end{equation*}
for a positive constant $C$ independent of $k$, with
$$K(R_k) = \Big(\big(\,R_k/4\,\big)^{\,\theta}\,\|\,f\,\|_{\,\frac{n}{p-\theta}\,,\,\mathbb{R}^n\backslash B_{R_k/4}(0)}\Big)^{\,\frac{1}{p-1}}. $$
By hypothesis \eqref{Kgoes0}, $K(R)\rightarrow0$ as $R\rightarrow\infty$, so that $K(R_k)\le\epsilon$, for all $k$ sufficiently large. Hence, using \eqref{2epsilon}, it follows
\begin{equation*}
\underset{S_{\,R_k}}{\,\sup\,v}\,\le\,C\,\epsilon\,,\,\text{ for all $k$ sufficiently large, }
\end{equation*}
and, consequently,
\begin{equation}\label{vbound}
\underset{\partial A(R_k,R_{k+1})}{\,\sup\,v}\,\le\,C\,\epsilon\,,\,\text{ for all $k$ sufficiently large. }
\end{equation}

\par In the sequence, we obtain a bound for $v$ on the interior of the annuli  $A_k:=B_{R_{k+1}}(0) \backslash B_{R_k}(0)$. For that we apply Corollary \ref{cortalenti} to $v$ on $\Omega = A_k$, with
$$f_k\, =\, f\,|\underset{A_k} \, .$$
As $|\Omega| = \omega_n\,(\,R_{k+1}^n-R_k^n\,)\le\omega_n\,R_{k+1}^n$ we obtain
\begin{equation}\label{vestimate}
\begin{split}
\,\underset{ A_k^{\sharp}}{\,\sup\,v^{\sharp}}
\le \,\underset{ \partial A_k}{\sup\,v} \,+\Big(\frac{1}{n\,\omega_n\,\delta}\Big)^{\frac{1}{p-1}}\,\int_{0}^{R_{k+1}} \rho^{-\frac{n-1}{p-1}}\bigg(\int_{B_\rho}f_k^{\sharp}(x)\,dx\,\bigg)^{\frac{1}{p-1}}\,d\rho,\\
\end{split}
\end{equation}
where $A_k^{\sharp}$ is the ball centered at 0 with the same measure as $A_k$. Let us split the integral above as
\begin{equation}\label{splitt}
\begin{split}
&\,\,\,\,\,\int_{0}^{R_{k+1}} \rho^{-\frac{n-1}{p-1}}\bigg(\int_{B_\rho}f_k^{\sharp}(x)\,dx\bigg)^{\frac{1}{p-1}}\,d\rho\\
=\,&\,\,\,\,\int_{0}^{1} \rho^{-\frac{n-1}{p-1}}\bigg(\int_{B_\rho}f_k^{\sharp}(x)\,dx\bigg)^{\frac{1}{p-1}}\,d\rho\,\,+\,\int_{1}^{R_{k+1}} \rho^{-\frac{n-1}{p-1}}\bigg(\int_{B_\rho}f_k^{\sharp}(x)\,dx\bigg)^{\frac{1}{p-1}}\,d\rho.\\
\end{split}
\end{equation}
By the H\"{o}lder inequality we have, for any $q\geq1$,
\begin{equation*}
\int_{B_\rho}f_k^{\sharp}(x)\,dx\,\le\,\bigg(\int_{B_\rho}\big(f_k^{\sharp}\big)^{q}(x)\,dx \bigg)^{1/q}\,\big(\,\omega_n\,\rho^n\,\big)^{\frac{1}{q'}},\\
\end{equation*}
and using that
\begin{equation*}
\begin{split}
\bigg(\int_{B_\rho}\big(\,f_k^{\sharp}\big)^{q}(x)\,dx \bigg)^{1/q}\,&\le\,\|f\|_{\,q\,,\,A_k},
\end{split}
\end{equation*}
we get
\begin{equation}\label{holder}
\int_{B_\rho}f_k^{\sharp}(x)\,dx\,\le\,\omega_{n}^{\,\frac{1}{q'}}\,\rho^{\,\frac{n}{q'}}\,\|f\|_{\,q\,,\,A_k}.
\end{equation}
Using this with $q$ chosen as $\,s=\frac{n}{p-\theta}>\frac{n}{p}\,$ we have for the first integral in \eqref{splitt}
\begin{equation*}\label{splitestimate1}
\begin{split}
\int_{0}^{1} \rho^{-\frac{n-1}{p-1}}\bigg(\int_{B_\rho}f_k^{\sharp}(x)\,dx\bigg)^{\frac{1}{p-1}}\,d\rho\,\,
&\le\,\, \omega_n^{\,\frac{1}{s'(p-1)}}\,\int_{0}^{1}\rho^{-\frac{n-s}{s(p-1)}}\,d\rho\,\,\|f\|_{\,s\,,\,A_k}^{\frac{1}{p-1}}\\
&=\,\omega_n^{\,\frac{1}{s'(p-1)}}\,\,\frac{s(p-1)}{sp-n}\,\,\|f\|_{\,s\,,\,A_k}^{\frac{1}{p-1}}.
\end{split}
\end{equation*}
For the second integral in \eqref{splitt}, we use \eqref{holder} with $\,q=r<n/p\,$ to get
\begin{equation*}\label{splitestimate2}
\begin{split}
\int_{1}^{R_{k+1}} \rho^{-\frac{n-1}{p-1}}\bigg(\int_{B_\rho}f_k^{\sharp}(x)\,dx\bigg)^{\frac{1}{p-1}}d\rho \,
&\le \, \omega_n^{\,\frac{1}{r'(p-1)}}\int_{1}^{R_{k+1}} \rho^{-\frac{n-r}{r(p-1)}}\,d\rho\,\,\|f\|_{\, r,\,A_k}^{\frac{1}{p-1}}\\
&\le \, \omega_n^{\,\frac{1}{r'(p-1)}}\, \bigg(\frac{r(p-1)}{rp-n} \rho^{\frac{rp-n}{r(p-1)}} \bigg)\bigg|_{1}^{R_{k+1}}  \! \!\|f\|_{\,r\,,\,A_k}^{\frac{1}{p-1}}\\[2pt]
&\le \, \omega_n^{\,\frac{1}{r'(p-1)}}\,\,\frac{r(p-1)}{n-pr}\, \,\|f\|_{\,r\,,\,A_k}^{\frac{1}{p-1}}.
\end{split}
\end{equation*}
Putting these estimates together in \eqref{splitt}, we obtain as a result from \eqref{vestimate} that
\begin{equation}\label{result}
\begin{split}
\underset{ A_k^{\sharp}}{\,\sup\,v^{\sharp}}\,\le\,\underset{ \partial A_k}{\,\sup\,v} \,+\,C\,\Big(\,\|f\|_{\,r\,,\,A_k}^{\,\,\frac{1}{p-1}}\,+\,\|f\|_{\,s\,,\,A_k}^{\,\,\frac{1}{p-1}}\,\Big),
\end{split}
\end{equation}
for some constant $C$ depending only on $n, p, \delta, \Gamma, r, s$. Now by the hypotheses \eqref{Lr}, \eqref{Ltheta}, we have
\begin{equation*}\label{result2}
\|f\|_{\,r\,,\, A_k}^{\,\,\frac{1}{p-1}}\,\,+\,\,\|f\|_{\,s\,,\,A_k}^{\,\,\frac{1}{p-1}}\,\le\,\epsilon\,,\,\text{ for all $k$ sufficiently large. }
\end{equation*}
Then using \eqref{vbound}, \eqref{result} yields
\begin{equation*}\label{result3}
\underset{ A_k^{\sharp}}{\,\sup\,v^{\sharp}}\,\le\,C\,\epsilon\,,\,\text{ for all $k$ sufficiently large,}
\end{equation*}
with a constant $C$ depending only on $n, p, \delta, \Gamma, r, s$. Now, since
$$\underset{ A_k}{\,\sup\,v}\,=\underset{ A_k^{\sharp}}{\,\sup\,v^{\sharp}},$$
it follows that
\begin{equation*}
\underset{ A_k}{\,\sup\,v}\,\le\,C\,\epsilon\,,\,\,\text{ for all $\,k\,$ sufficiently large.}
\end{equation*}
Hence
\begin{equation*}
v(x)\,\le\,C\,\epsilon\,,\,\text{ for all $|x|$ sufficiently large,}
\end{equation*}
and then, by definition of $v$, we have
\begin{equation*}
u(x)-m\,\le\,C\,\epsilon\,,\,\text{ for all $|x|$ sufficiently large.}
\end{equation*}
Since $\epsilon$ is arbitrary it follows that
\begin{equation*}
\underset{|x|\rightarrow\infty}{\limsup}\,u\,\le\,m,
\end{equation*}
proving that $\,\underset{|x|\rightarrow\infty}{\lim}\,u(x)\,=\,m\,$.
\hfill$\square$

\

The next result provides an estimate for solutions of \eqref{nonhomogeneousPr} in the case (b) of Theorem \ref{TeorPrinc} in terms of the ``boundary data" and $f$. If $f = 0$, it is a kind of maximum principle.

\begin{Cor} Consider $p<n$ and $u\in C^{1}(\mathbb{R}^{n}\setminus B_1)$ a bounded weak solution of $$-{\rm div}\big(|\nabla u|^{p-2}A(|\nabla u|)\nabla u\big)\,=\,f $$ in $\mathbb{R}^{n}\setminus B_1$. If $f$ satisfies the assumptions \eqref{Lr}, \eqref{Ltheta}, and \eqref{Kgoes0}, then
\begin{equation*}
\begin{split}
\underset{ \mathbb{R}^n \setminus B_{R}(0)}{\,\sup\,u}\,\le\,\max\lbrace \underset{ \partial B_{R}(0)}{\,\sup\,u}, m\rbrace \,+\,C\,\Big(\,\|f\|_{\,r\,,\,\mathbb{R}^n \setminus B_{R}(0)}^{\,\,\frac{1}{p-1}}\,+\,\|f\|_{\,\frac{n}{p-\theta}\,,\,\mathbb{R}^n \setminus B_{R}(0)}^{\,\,\frac{1}{p-1}}\,\Big),
\end{split}
\end{equation*}
and
\begin{equation*}
\begin{split}
\underset{ \mathbb{R}^n \setminus B_{R}(0)}{\,\inf\,u}\,\geq\,\min\lbrace \underset{ \partial B_{R}(0)}{\,\inf\,u}, m\rbrace \,-\,C\,\Big(\,\|f\|_{\,r\,,\,\mathbb{R}^n \setminus B_{R}(0)}^{\,\,\frac{1}{p-1}}\,+\,\|f\|_{\,\frac{n}{p-\theta}\,,\,\mathbb{R}^n \setminus B_{R}(0)}^{\,\,\frac{1}{p-1}}\,\Big),
\end{split}
\end{equation*}
for $R >1$, where $m = \displaystyle \lim_{ |x| \rightarrow \infty } u(x)$ and $C = C(n, p, \delta, \Gamma, r, \theta)$. In particular, if $$\underset{\partial B_{R}(0)}{\,\inf\,u} \leq m \leq \underset{\partial B_{R}(0)}{\,\sup\,u},$$ we have
\begin{equation*}
\begin{split}
\underset{ \mathbb{R}^n \setminus B_{R}(0)}{\,\operatorname{osc} \,u}\,\leq\, \underset{ \partial B_{R}(0)}{\,\operatorname{osc} \,u} +\,C\,\Big(\,\|f\|_{\,r\,,\,\mathbb{R}^n \setminus B_{R}(0)}^{\,\,\frac{1}{p-1}}\,+\,\|f\|_{\,\frac{n}{p-\theta}\,,\,\mathbb{R}^n \setminus B_{R}(0)}^{\,\,\frac{1}{p-1}}\,\Big).
\end{split}
\end{equation*}
\end{Cor}

\begin{proof}
Considering the equation \eqref{result}, using that $v= u - m +\epsilon$, and $$\underset{ A(R,\tilde{R})}{\,\sup\,v}\,=\underset{ A(R,\tilde{R})^{\sharp}}{\,\sup\,v^{\sharp}},$$
where $1 < R < \tilde{R}$, we obtain
\begin{equation}\label{resultu}
\begin{split}
\underset{ A(R,\tilde{R})}{\,\sup\,u}\,\le\,\underset{ \partial A(R,\tilde{R})}{\,\sup\,u} \,+\,C\,\Big(\,\|f\|_{\,r\,,\,A(R,\tilde{R})}^{\,\,\frac{1}{p-1}}\,+\,\|f\|_{\,\frac{n}{p-\theta}\,,\,A(R,\tilde{R})}^{\,\,\frac{1}{p-1}}\,\Big).
\end{split}
\end{equation}
Now, observe that
$$\underset{ \partial A(R,\tilde{R})}{\,\sup\,u} = \max \lbrace \underset{ \partial B_{R}(0)}{\,\sup\,u}, \, \underset{ \partial B_{\tilde{R}}(0)}{\,\sup\,u} \rbrace.$$
Then, making $\tilde{R} \rightarrow \infty$ in \eqref{resultu}, we have
\begin{equation*}
\begin{split}
\underset{ \mathbb{R}^n \setminus B_{R}(0)}{\,\sup\,u}\,\le\,\max\lbrace \underset{ \partial B_{R}(0)}{\,\sup\,u}, m\rbrace \,+\,C\,\Big(\,\|f\|_{\,r\,,\,\mathbb{R}^n \setminus B_{R}(0)}^{\,\,\frac{1}{p-1}}\,+\,\|f\|_{\,\frac{n}{p-\theta}\,,\,\mathbb{R}^n \setminus B_{R}(0)}^{\,\,\frac{1}{p-1}}\,\Big).
\end{split}
\end{equation*}
The result for the infimum follows considering the solution $-u$.

If  $\underset{\partial B_{R}(0)}{\,\inf\,u} \leq m \leq \underset{\partial B_{R}(0)}{\,\sup\,u}$, we obtain $$\max\lbrace \underset{ \partial B_{R}(0)}{\,\sup\,u}, m\rbrace = \underset{ \partial B_{R}(0)}{\,\sup\,u} \quad \rm{ and} \quad \min\lbrace \underset{ \partial B_{R}(0)}{\,\inf\,u}, m\rbrace = \underset{ \partial B_{R}(0)}{\,\inf\,u},$$
and the result for the oscillation follows when we subtract the estimates.
\end{proof}

\section{Limit at infinity for $f$ that decays to zero}

\par The goal of this section is to prove $(a)$, for $p<n$, and $(c)$ of Theorem \ref{TeorPrinc}. So, we assume that $A$ satisfies \eqref{A_conditions2} and $f$ decays to zero as in \eqref{f}. This condition allows us to improve Lemmas 1 and 2 getting better estimates for supersolutions as their domains of definition go farther away from the origin.
\begin{Lem1'}
Assume $p> n$, $f$ is bounded and satisfy the condition \eqref{f}. Then, for any $x_0\in S_{2R}$, $R>1$, there exists a family of radially symmetric supersolutions  $ v_{a,x_0}, a\ge0 $ of \eqref{nonhomogeneousPr} in $B_{R}(x_0)\backslash\{x_{0}\}$
satisfying
\begin{equation}\label{lem1'_est}
\bigg(\frac{1}{L}\bigg)^{\frac{1}{p-1}}a\,\frac{|x-x_0|^{\alpha}}{\alpha} \le v_{a,x_0}(x) \le  \bigg(\frac{1}{\delta}\bigg)^{\frac{1}{p-1}}\! \! \left(a+ \left(\frac{C_f}{n}\right)^{\frac{1}{p-1}} \! \! R^{-\frac{p-n+\epsilon}{p-1}} \right) \! \frac{|x-x_0|^{\alpha}}{\alpha},
\end{equation}
for $a \ge \displaystyle  0$, with $\alpha = \frac{p-n}{p-1}$.
\label{lema1'}
\end{Lem1'}
\begin{proof}
This comes from Lemma 1 and $\|f\|_{L^{\infty}(B_R(x_0))}\leq C_f\,R^{-p-\epsilon}$.
\end{proof}
The second improvement concerns about supersolutions defined on the complement of large balls.
\begin{Lem2'}\label{lemma2-prime}
Assume $p\ge n$ and $f$ satisfy the condition \eqref{f}. Then, for all $R > 1$, there exists a family of radially symmetric supersolutions $v_a, a\ge0$ of \eqref{nonhomogeneousPr} in $\mathbb{R}^n \backslash B_R(0)$ satisfying
\\[5pt]  (a) $v_a(R)=0$ and $v_a(r)$ is increasing in $[R,+\infty)$ for any $a\ge 0$;
\\ (b) $v_a$ is unbounded in $[R,+\infty)$ for $a > 0$; indeed, there exists a positive constant $c_0=c_0(n,p,L)$  such that
 $$v_a(r)\, \ge\, c_0\,a\, (r^{\alpha} - R^{\alpha}) \quad {\rm for} \quad r \ge R, \quad {\rm if} \quad p > n$$
 $$v_a(r)\, \ge\, c_0\,a\, (\log r - \log R) \quad {\rm for} \quad r \ge R, \quad {\rm if} \quad p = n;$$
 (c) $v_0$ is bounded in $[R,+\infty)$; indeed, there exists $C_0=C_0(n,p,\epsilon,C_f,\delta) >0$ such that
 $$v_0(r) \,\le\, C_0 \,(R^{-\frac{\epsilon}{p-1}} - r^{-\frac{\epsilon}{p-1}}) \quad {\rm for } \quad r \ge R;$$
 (d) $v_a(r) \to v_0(r)$ as $a \to 0$ for any $r \in [R,+\infty)$.
\label{upAndLowEstForVa-new2}
\end{Lem2'}
\begin{proof}
This follows the same lines of the proof of Lemma 2. In this case, the function $g$ can be taken as
\begin{equation}
g(r)\, =\,\frac{C_f}{r^{\,p+\epsilon}}\,\,,\,\,\,r\,\ge\,R \,.
\end{equation}
 Performing the integrations from $R$ onwards we obtain the supersolutions
\begin{equation}
 v(r) = \int_{R}^{r}\varphi^{-1}\left( \,\frac{\frac{C_f}{p-n+\epsilon}\big(\,t^{n-p-\epsilon}-R^{n-p-\epsilon}\,\big)+C}{t^{n-1}}\,\,\right)dt\,\,
\end{equation}
for $r\geq R$, with $C\,=\,|v'(R)|^{\,p-2}A\big(\,|v'(R)|\,\big)\,v'(R)$.
Choosing
\begin{equation}
C = \frac{C_f}{p-n+\epsilon} R^{n-p-\epsilon} + a^{\,p-1}\,,
\end{equation}
where $a \ge 0$, we have
\begin{equation}
 v_a(r) = \int_{R}^{r}\varphi^{-1}\left( \,\frac{\frac{C_f}{p-n+\epsilon} t^{n-p-\epsilon}+a^{p-1}}{t^{n-1}}\,\,\right)dt\, ,
\end{equation}
that is well defined and satisfies $(a)$ and $(d)$.
Then using the estimates \eqref{phi} for $\varphi^{-1}$, we get
\begin{equation*}
\begin{split}
    v_a(r) \,\geq&\, \bigg(\frac{1}{L}\bigg)^{\frac{1}{p-1}}\int_{R}^{r}\left(\frac{C_f}{p-n+\epsilon} t^{\,n-p-\epsilon} +a^{\,p-1} \right)^{\frac{1}{p-1}}\,t^{\,-\frac{n-1}{p-1}}\,dt\,\\
 \geq&\, \bigg(\frac{1}{L}\bigg)^{\frac{1}{p-1}}\int_{R}^{r} \,a\,t^{\,-\frac{n-1}{p-1}}\,dt\,\\
\end{split}
\end{equation*}
from which it follows
\begin{equation}\label{2'p>n}
    v_a(r) \,\geq\, \bigg(\frac{1}{L}\bigg)^{\frac{1}{p-1}}\,a\,\frac{(\,r^{\,\alpha}-R^{\,\alpha}\,)}{\alpha}
\end{equation}
in case $p>n$ and
\begin{equation}\label{2'p=n}
\begin{split}
    v_a(r) \,\geq&\, \bigg(\frac{1}{L}\bigg)^{\frac{1}{p-1}}\,a\,(\,\log r - \log R\,)
\end{split}
\end{equation}
if $p=n$. For the upper bound for $a=0$ we have for $p\ge n$
\begin{equation}\label{bound0}
\begin{split}
 v_0(r) \, &\le \, \bigg(\frac{C_f}{\delta(p-n+\epsilon)}\bigg)^{\,\frac{1}{p-1}}\,\int_{R}^{r} \left( \,\frac{\,t^{n-p-\epsilon}\,}{t^{n-1}}\,\,\right)^{\,\frac{1}{p-1}}\,dt\\
 &\le \,\bigg(\frac{C_f}{\delta(p-n+\epsilon)}\bigg)^{\,\frac{1}{p-1}\,}\,\Big(\,\frac{p-1}{\epsilon}\,\Big)\,(R^{-\frac{\epsilon}{p-1}} - r^{-\frac{\epsilon}{p-1}}).\\
\end{split}
\end{equation}
\end{proof}
\par Next result is a kind of extension of estimates obtained in \cite{S1} (or Proposition 3 of \cite{BSZ}) for the nonhomogeneous case. First we fix some notation. For a function $u$, let
$$ m_{R}=\inf_{S_{R}} u \quad \text{and} \quad M_{R}=\sup_{S_{R}} u \quad \text{for} \quad R >0,$$ where $S_{R}=\{ x \in \mathbb{R}^n \, : \, \| x \|= R \}$.
The oscillation of $u$ on $S_{R}$ is defined as
$$\underset{S_{R}\,}{\,osc\,u} = M_{R} - m_{R}.$$

\begin{Prop}\label{liouville}
 Let $u\in W^{1,p}_{loc}(\mathbb{R}^n\backslash \overline{B_1})$ be a bounded weak solution of \eqref{nonhomogeneousPr}, with $f$ satisfying condition \eqref{f}. Then, in case $p\ge n$, for all $R \ge R_0$,
\begin{equation} \label{almostmaximumprinciple}
 m_R - C_{0}R^{-\frac{\epsilon}{p-1}} \le u(x) \le M_{R} + C_{0}R^{-\frac{\epsilon}{p-1}} \quad \text{for} \quad x \in \mathbb{R}^n \backslash B_R ,
\end{equation}
where $ C_{0}=\left(\frac{C_f}{\delta(p-n+\epsilon)}\right)^{\frac{1}{p-1}}\left(\frac{p-1}{\epsilon}\right)$. In particular, if $R_0=1$, we have
the following global bound for $u$:
$$ \inf_{S_1} u - C_{0} \le u \le \sup_{S_1} u + C_{0}.$$
\end{Prop}
\begin{proof}
Assume with no loss of generality that \eqref{f} holds for all $|x|\ge1$. Suppose now that the weak solution $u$ in question satisfies \eqref{alpha-subgrowth}. For $R\geq 1$, consider the family of radially symmetric supersolutions $v_{a}$ given by Lemma 2'. Hence the second property of $v_a$ and since $u$ is bounded (or satisfies the properties \eqref{alpha-subgrowth} or \eqref{alpha-subgrowth2}), we obtain for each $a > 0$ a $R_{a}> R$ such that
$$ M_{R}+ v_{a}(|x|)\geq u(|x|) \quad \text{for all} \quad |x| \geq R_{a}\,.$$
Consequently, the function $w_{a}(r):= M_{R}+ v_{a}(r),$ $r\geq R$, lies above $u$ on the boundary of the annulus $B_{R_{a}}\setminus B_{R}$. Then, by the comparison principle, $w_{a}\geq u$ on $B_{R_{a}}\setminus B_{R}$, that is, $w_{a}\geq u$ on $\mathbb{R}^{n}\setminus B_{R}$. Then, for $x\in\mathbb{R}^{n}\setminus B_{R}$, the third and fourth properties of $v_a$ in Lemma 2' imply that
\begin{equation}\label{bound}
u(x)\leq \lim_{a\rightarrow 0} w_{a}(|x|) = M_{R}+ v_{0}(|x|)< M_{R} + C_{0}R^{-\frac{\epsilon}{p-1}}\,.
\end{equation}
From \eqref{bound0}, we can see that $C_0$ is given by
$$C_{0}=\left(\frac{C_f}{\delta(p-n+\epsilon)}\right)^{\frac{1}{p-1}}\left(\frac{p-1}{\epsilon}\right).$$
Analogously, we can prove that $ u(x) \ge m_R - C_{0}R^{-\frac{\epsilon}{p-1}}$ for $x \in \mathbb{R}^n \setminus B_R$.\\
\end{proof}

\begin{Cor}
Under the hypotheses of Theorem \ref{TeorPrinc}, there exist the limits
 \begin{equation}
 \lim_{R\rightarrow+\infty} m_{R} \quad \text{and} \quad  \lim_{R\rightarrow+\infty} M_{R}.
 \end{equation}
\end{Cor}

\

\par \textit{2. Proof of item (a) of Theorem \ref{TeorPrinc}:}

\

We prove this only for the case that $u$ is bounded from below, since the argument for $u$ that is bounded from above is analogous and, combining these two conclusions, we can prove the result when $u$ is bounded. Furthermore, we can suppose w.l.g. that $u$ is nonnegative, since $u+c$ is also a solution if $c$ is a constant. \\
\par Let $u$ be a nonnegative weak solution of \eqref{nonhomogeneousPr} on $\mathbb{R}^n\setminus \overline{B_1}$ and set $ m\,=\,\underset{|x|\rightarrow\infty}{\liminf}\,u$. If $m = +\infty$, there is nothing to prove, so we assume $m<+\infty$. For a given $\varepsilon >0$, let $R_0>0$ be such that
\begin{equation*}
u(x)\,>\,m-\varepsilon\,\,\,\text{for all}\,\,x\,\,\text{such that}\,\,|x|\ge R_0\,
\end{equation*}
so that the function
\begin{equation*}
v\,=\,u-m+\varepsilon\,
\end{equation*}
is a positive solution on $\mathbb{R}^n\setminus B_{R_0}$. We pick up a sequence of points $(x_k)$, with $|x_k|\rightarrow \infty$, $R_0<|x_k|<|x_{k+1}|$, such that
\begin{equation*}
u(x_k)\,\le\,m+\epsilon
\end{equation*}
and, consequently,
\begin{equation}\label{2epsilon-Again}
v(x_k)\,\le\,2\epsilon\,.
\end{equation}
Now let $R_k = |x_k|$, $S_{R_k} = \partial B_{R_k}(0)$. By applying the Corollary \ref{SRHarnack} to $v$ we get
\begin{equation*}
\underset{S_{\,R_k}}{\,\sup\,v}\,\le\,C\,\Big(\,\underset{S_{\,R_k}}{\,\inf\,}\,v + R_k^{-\frac{\epsilon}{p-1}}\,\Big)
\end{equation*}
for a positive constant $C$ independent of $k$. Hence, since $R^{-\frac{\epsilon}{p-1}}\rightarrow0$ as $R\rightarrow\infty$, by \eqref{2epsilon-Again} it follows that
\begin{equation*}
\underset{S_{\,R_k}}{\,\sup\,v}\,\le\,C\,\epsilon\,,\,\,\,\text{ for all $k$ sufficiently large }
\end{equation*}
and, consequently,
\begin{equation}\label{vbound-Again}
\underset{\partial A(R_k,R_{k+1})}{\,\sup\,v}\,\le\,C\,\epsilon\,,\,\,\,\text{ for all $k$ sufficiently large. }
\end{equation}
We then proceed to bound $v$ on the interior of each annuli with the use of barriers. By Lemma 2', $(c)$, for each $k$, we have a positive supersolution $v_{\,0}=v_{\,0,\,k}$ in $\mathbb{R}^n\setminus \overline{B_{R_k}}$ satisfying
$$v_{\,0,\,k}\le C_0 \,R_k^{-\frac{\epsilon}{p-1}}\,.$$
Hence, the function
$$w_k(x)\,:=\,C\,\varepsilon \,+\,v_{\,0,\,k}(x)$$
where $C$ is the constant from \eqref{vbound-Again}, is such that, for any natural $l> k$, $w_k \ge v$ in $\partial A(R_k,R_{l})$. The comparison principle then gives $w_k \ge v$ in $ A(R_k,R_{l})$, from which follows the bound
$$v\,\le\,w_k\,\le\,C\,\varepsilon \,+ C_0\,R_k^{-\frac{\epsilon}{p-1}} \quad\text{ in }\,\,\mathbb{R}^n\setminus \overline{B_{R_k}}$$
and, by redefining the constant $C$,
$$v(x)\,\le\, C\,\epsilon\,,\,\,\,\text{ for all $|x|$ sufficiently large.}$$
Hence, by definition of $v$, we have
\begin{equation*}
u(x)-m\,\le\,C\,\epsilon\,,\,\,\,\text{ for all $|x|$ sufficiently large,}
\end{equation*}
from which follows that
\begin{equation*}
\underset{|x|\rightarrow\infty}{\limsup}\,u\,\le\,m \, ,
\end{equation*}
proving that $\,\underset{|x|\rightarrow\infty}{\lim}\,u(x)\,=\,m\,$. \hspace*{\fill}$\square$\\

\

\par For the last item of Theorem \ref{TeorPrinc}, we have to control of the oscillation of u:

\begin{Lem}
\par  Let $u\in C^{1}(\mathbb{R}^{n}\setminus B_1)$ a bounded solution of \eqref{nonhomogeneousPr} in $\mathbb{R}^{n}\setminus B_1$ and assume $f$ satisfy condition \eqref{f} and $p>n$. Then, there are constants $0<C<1$ and $K\geq0$, such that
\begin{equation}\label{osc}
\underset{S_{2R}\,}{\,osc\,u} \leq C\,\big(\,\underset{S_{R}\,}{\,osc\,u} + K.R^{-\frac{\epsilon}{p-1}}\,\big)
\end{equation}
for all $R\geq 1$.
\end{Lem}

\begin{proof}
Given $R\geq 1$, let $x_{1}\in S_{2R}$ such that $u(x_{1})=m_{2R}$. For $x'\in S_{2R}$, let $\gamma\subset S_{2R}$ be the shortest geodesic joining $x_{1}$ to $x'$. By a recursive process starting at $x_{1}$, we obtain estimates for $u$ on successive balls with centers in $\gamma$, up to $x'$. \\
\par In the first step, we set $u_{1}=u(x_{1})$ and define for $x\in B_{R}(x_{1})$
$$w_{1}(x)=w_{1}(r)=u_{1}+v_{a_{1},x_1}(r), \;\; r=|x-x_{1}|\leq R,$$
where $v_{a_{1},x_1}$ is a supersolution in $ B_{R}(x_{1})$ given by Lemma 1'. We will chose $a_{1}$ so that
\begin{equation}
w_{1}(R)\geq M_{R}+C_{0}R^{-\frac{\epsilon}{p-1}}.
\label{w1-MR-C0}
\end{equation}
For this, using the lower estimate for $v_{a_1,x_1}$ in Lemma 1', it is sufficient to require that
$$u_{1}+\left(\frac{1}{L}\right)^{\frac{1}{p-1}}\,a_{1}\,\frac{R^{\alpha}}{\alpha}\,\geq\, M_{R}+C_{0}R^{-\frac{\epsilon}{p-1}},$$
which is equivalent to
\begin{equation*}
    a_{1}\,\geq\, \alpha R^{-\alpha}\left(\,\frac{1}{L}\,\right)^{-\frac{1}{p-1}}\left(M_{R}+C_{0}R^{-\frac{\epsilon}{p-1}}-u_{1}\right)\,.
\end{equation*}
Hence, setting
\begin{equation}\label{a1}
   a_{1}\,=\, \alpha R^{-\alpha}\left(\,\frac{1}{L}\,\right)^{-\frac{1}{p-1}}\left(M_{R}+C_{0}R^{-\frac{\epsilon}{p-1}}-u_{1}\right),
    \end{equation}
we conclude that \eqref{w1-MR-C0} is satisfied. From this and Proposition \ref{liouville}, we get that $w_{1}\geq u$ on $\partial B_{R}(x_{1})$ and, by the comparison principle,
\begin{equation}\label{bound2}
w_{1}\geq u\quad\mbox{on }B_{R}(x_{1})\,.
\end{equation}
Next, we wish to find some radius $R_{1}\leq R$ such that
$$w_1(r)\,<\,M_{R}+C_{0}R^{-\frac{\epsilon}{p-1}}+\frac{1}{\alpha}\left(\,\frac{C_f}{nL}\,\right)^{\frac{1}{p-1}}\,R^{-\frac{\epsilon}{p-1}}\,\,\,\text{ for all }\, r\,\le\, R_1\,. $$
In view of the upper estimate in Lemma 1' we have
\begin{equation}\label{w1}
w_{1}(r)=u_{1}+v_{a_{1},x_{1}}(r)\leq u_{1}+\bigg(\frac{1}{\delta}\bigg)^{\frac{1}{p-1}}\left(a_1+ \bigg(\frac{C_f}{n}\bigg)^{\frac{1}{p-1}} R^{-\frac{p-n+\epsilon}{p-1}}\, \right)\frac{r^{\alpha}}{\alpha}\,.
\end{equation}
Hence, it is enough to find $R_{1}\leq R$ such that
\begin{equation}\label{r1}
 \begin{split}
 &\,\,\,\,\,\,u_{1}+ \bigg(\frac{1}{\delta}\bigg)^{\frac{1}{p-1}}\left(a_1+ \bigg(\frac{C_f}{n}\bigg)^{\frac{1}{p-1}} R^{-\frac{p-n+\epsilon}{p-1}}\, \right) \frac{R_1^{\alpha}}{\alpha}
 \\ \le&\,\,M_{R}+C_{0}R^{-\frac{\epsilon}{p-1}}+\frac{1}{\alpha}\left(\,\frac{C_f}{nL}\,\right)^{\frac{1}{p-1}}\,R^{-\frac{\epsilon}{p-1}}\,.
\end{split}
\end{equation}
Replacing the expression of $a_{1}$ and solving for $R_{1}$ gives
\begin{equation*}
  R_{1} <\,  \left(\frac{\delta}{L}\right)^{\frac{1}{(p-1)\alpha}} R
    =\,  \left(\frac{\delta}{L}\right)^{\frac{1}{p-n}} R,
\end{equation*}
so we take
\begin{equation}\label{lambda}
    R_{1} =  \lambda R, \;\;\;\lambda= \frac{1}{2}\left(\frac{\delta}{L}\right)^{\frac{1}{p-n}}.
\end{equation}
To the next step, motivated by \eqref{w1}, we define
\begin{equation*}
 u_{2}\,=u_{1}+ \bigg(\frac{1}{\delta}\bigg)^{\frac{1}{p-1}}\left(a_1+ \bigg(\frac{C_f}{n}\bigg)^{\frac{1}{p-1}} R^{-\frac{p-n+\epsilon}{p-1}}\, \right) \frac{(\,\lambda R\,)^{\alpha}}{\alpha}
\end{equation*}
which is the upper bound for $w_1$ in $B_{\lambda R}(x_1)$.
We then take
$$x_{2}\in \gamma\cap \partial B_{\lambda R}(x_{1})$$
the closest point to $x'$ in this intersection and define as before
$$w_{2}(r)=u_{2}+v_{a_{2},x_2}(r),\;\mbox{ for } r=|x-x_{2}|\leq R$$
with $v_{a_2,x_2}$ being the supersolution in $B_R(x_2)$ given in Lemma 1'. Analogously to the previous step, the choice
\begin{equation*}
a_{2}\, =\, \alpha R^{-\alpha}\left(\,\frac{1}{L}\,\right)^{-\frac{1}{p-1}}\left(M_{R}+C_{0}R^{-\frac{\epsilon}{p-1}}-u_{2}\right)
\end{equation*}
ensures that
\begin{equation*}
w_{2}\,\geq\, u\;\mbox{ on }B_{R}(x_{2})\,.
\end{equation*}
Also, the same calculation carried out in the first step shows
\begin{equation*}
w_{2}(r) < M_{R}+C_{0}R^{-\frac{\epsilon}{p-1}}+\frac{1}{\alpha} \left( \frac{C_f}{nL}\right)^{\frac{1}{p-1}}R^{-\frac{\epsilon}{p-1}}, \;\mbox{ for } r\leq \lambda R
\end{equation*}
for $\lambda$ already defined in \eqref{lambda}.
Next we take
$$ u_{3}=u_{2}+\left(\frac{1}{\delta}\right)^{\frac{1}{p-1}}\,\left(a_2+ \left(\frac{C_f}{n}\right)^{\frac{1}{p-1}} R\,^{-\frac{p-n+\epsilon}{p-1}} \right)\,\frac{(\,\lambda R\,)^{\alpha}}{\alpha}$$
and
$$x_{3}\in \gamma\cap \partial B_{\lambda R}(x_{2})$$
the closest point to $x'$ in this intersection, and repeat the procedure.
After $k-1$ steps, we reach at some point $x_{k}\in\gamma$, having defined
\begin{equation}\label{recurrence}
u_{k}=u_{k-1}+\,\left(\frac{1}{\delta}\right)^{\frac{1}{p-1}}\,\left(\,a_{k-1}+ \left(\frac{C_f}{n}\right)^{\frac{1}{p-1}} R\,^{-\frac{p-n+\epsilon}{p-1}}\, \right)\,\frac{(\,\lambda R\,)^{\alpha}}{\alpha}
\end{equation}
with
\begin{equation}\label{ak-1}
a_{k-1}\,=\, \alpha R^{-\alpha}\left(\,\frac{1}{L}\,\right)^{-\frac{1}{p-1}}\left(M_{R}+C_{0}R^{-\frac{\epsilon}{p-1}}-u_{k-1}\right)
\end{equation}
and
\begin{equation*}
u\leq u_{k}< M_{R}+C_{0}R^{-\frac{\epsilon}{p-1}}+\frac{1}{\alpha} \left( \frac{C_f}{nL}\right)^{\frac{1}{p-1}}R^{-\frac{\epsilon}{p-1}}\;\;\mbox{ in  }B_{\lambda R}(x_{k-1})\,.
\end{equation*}
From \eqref{recurrence}, \eqref{ak-1}, we obtain the recurrence
\begin{equation*}
\begin{split}
u_k\,&=\,u_{k-1}\left(\,1-\lambda^{\alpha}\left(\,\frac{L}{\delta}\,\right)^{\frac{1}{p-1}}\,\right)\,\,\,\,+ \\ &\,\,+\,\,\,\,\lambda^{\alpha}\left(\,\frac{L}{\delta}\,\right)^{\frac{1}{p-1}}\left(\,M_R+C_0\,R^{-\frac{\epsilon}{p-1}}+
 \frac{1}{\alpha}\left(\,\frac{C_f}{nL}\,\right)^{\frac{1}{p-1}}R^{-\frac{\epsilon}{p-1}}\,\right)\,,\\
 u_1\,&=\,m_{2R}
\end{split}
\end{equation*}
from which we determine
\begin{equation*}
\begin{split}
&u_{k}\;=\;M_{R}+C_{0}R^{-\frac{\epsilon}{p-1}}+\frac{1}{\alpha} \left( \frac{C_f}{nL}\right)^{\frac{1}{p-1}}R^{-\frac{\epsilon}{p-1}}\;\\
&-\left(  M_{R}+C_{0}R^{-\frac{\epsilon}{p-1}}-m_{2R}+\frac{1}{\alpha} \left( \frac{C_f}{nL}\right)^{\frac{1}{p-1}}R^{-\frac{\epsilon}{p-1}}\right)\left( 1- \lambda^{\alpha}\left(\frac{L}{\delta}\right)^{\frac{1}{p-1}}\right)^{k-1}
\end{split}
\end{equation*}
This comes from the fact that the solution to the recurrence relation
$$\,\,a\,u_k + b\,u_{k-1} + c\, =\, 0 $$
 is given by
$$ u_k\,=\,-\frac{c}{a+b}+\Big(\,u_1+\frac{c}{a+b}\,\Big)\Big(-\frac{b}{a}\Big)^{k-1}\,.$$

We stop the process when $\gamma$ is fully covered by the balls $B_{\lambda R}(x_k)$, which happens when the point $x_{k}$ reaches a distance to $x'$ less than $\lambda R$.
As the length of $\gamma$ is less than $2R \, \pi$ and each ball covers a segment over $\gamma$ with length greater than $\lambda R$, we see the number $l$ of balls needed to cover $\gamma$ is independent of $R$ and always less than $2\pi/\lambda +1$. Now, since $x'\in B_{\lambda R}(x_{l})$ and $u\le w_{l}\le u_{l+1}$ in $B_{\lambda R}(x_{l})$
it follows that
\begin{equation*}
\begin{split}
u(x') \, \leq \, u_{l+1} =& \,M_{R}+C_{0}R^{-\frac{\epsilon}{p-1}}+\frac{1}{\alpha} \left( \frac{C_f}{nL}\right)^{\frac{1}{p-1}}R^{-\frac{\epsilon}{p-1}}\\
&\,-\left( \,M_{R}+C_{0}R^{-\frac{\epsilon}{p-1}}-m_{2R}+\frac{1}{\alpha} \left( \frac{C_f}{nL}\right)^{\frac{1}{p-1}}R^{-\frac{\epsilon}{p-1}}\,\right)\,c
\end{split}
\end{equation*}
for
\begin{equation*}
c=\left( 1- \lambda^{\alpha}\left(\frac{L}{\delta}\right)^{\frac{1}{p-1}}\right)^{l}=\left(\,1 - \frac{1}{2^{\alpha}}\,\right)^l \leq \left( \frac{1}{2} \right)^{l} <1\,.
\end{equation*}
Being $x'$ arbitrary, we have $M_{2R}\leq u_{l+1}$ and, therefore,
\begin{equation*}
M_{2R}-m_{2R}  \leq  \,\left(\, M_{R}+C_{0}R^{-\frac{\epsilon}{p-1}}-m_{2R}+\frac{1}{\alpha} \left( \frac{C_f}{nL}\right)^{\frac{1}{p-1}}R^{-\frac{\epsilon}{p-1}}\,\right)(1-c).
\end{equation*}
Then using that $m_{2R}\geq m_{R} - C_{0}R^{-\frac{\epsilon}{p-1}} $, by Proposition \ref{liouville}, it comes
\begin{equation*}
M_{2R}-m_{2R} \leq  \left( M_{R}-m_{R}+2C_{0}R^{-\frac{\epsilon}{p-1}}+\frac{1}{\alpha} \left( \frac{C_f}{nL}\right)^{\frac{1}{p-1}}R^{-\frac{\epsilon}{p-1}}\right)(1-c),
\end{equation*}
that is,
\begin{equation}\label{osc5}
\underset{S_{2R}\,}{\,osc\,u}\,\le \,C\;\Big(\,\underset{S_{R}\,}{\,osc\,u} + K R^{-\frac{\epsilon}{p-1}}\,\Big),
\end{equation}
where
\begin{equation*}
C\,=\,1-c\,,\,\,\,\,\,\,\,K=2C_{0}+\frac{1}{\alpha} \left( \frac{C_f}{nL}\right)^{\frac{1}{p-1}}\,.
\end{equation*}
\end{proof}

\par \textit{3. Proof of item (c) of Theorem \ref{TeorPrinc}:}\\
\par From \eqref{osc5}, the proof follows by a standard argument. Indeed iterating it, we obtain
\begin{equation*}
\begin{split}
\underset{S_{2^{k}R}\,}{\,osc\,u} & \,\leq\,  C^k\,\Big(\,\underset{S_{R}\,}{\,osc\,u} + K R^{-\frac{\epsilon}{p-1}} \sum_{j=0}^{k-1}\bigg( \frac{2^{-\frac{\epsilon}{p-1}}}{C}\bigg)^{j}\,\,\Big)\\
\end{split}
\end{equation*}
We can assume without loss of generality that $C > 2^{-\frac{\epsilon}{p-1}}$.
Then we have
\begin{equation*}
\sum_{j=1}^{k}\bigg( \frac{2^{-\frac{\epsilon}{p-1}}}{C}\bigg)^{j}\,\,\le\,\,\frac{1}{1-\frac{2^{-\frac{\epsilon}{p-1}}}{C}}\,\,\le\,\,\frac{1}{C-2^{-\frac{\epsilon}{p-1}}}
\end{equation*}
and we get
\begin{equation}\label{oscb}
\begin{split}
\underset{S_{2^{k}R}\,}{\,osc\,u} & \,\leq\, C^k\,\bigg(\,\underset{S_{R}\,}{\,osc\,u}\,+ \frac{K\, R^{-\frac{\epsilon}{p-1}}}{\;C-2^{-\frac{\epsilon}{p-1}}\;}\,\bigg)\,\,\,\mbox{ for all }\;\;R\geq1.
\end{split}
\end{equation}
In particular,
\begin{equation}\label{oscb1}
\underset{S_{2^{k}}\,}{\,osc\,u} \,\leq\, C^k\,\bigg(\,\underset{S_{1}\,}{\,osc\,u}\,+ \frac{K}{\;C-2^{-\frac{\epsilon}{p-1}}\;}\,\bigg)
\end{equation}
Now, considering $x\in \mathbb{R}^{n}\backslash B_{1}$, there is some integer $k$ such that
\begin{equation}\label{|x|}
2^{k}\leq |x| \leq 2^{k+1}
\end{equation}
From \eqref{almostmaximumprinciple} and the assumption that $C> 2^{-\frac{\epsilon}{p-1}}$, we have
\begin{equation*}
\begin{split}
\underset{S_{|x|}\,}{\,osc\,u}\;\leq & \; \underset{S_{2^{k}}\,}{\,osc\,u}+2C_{0} \left( 2^{k} \right)^{-\frac{\epsilon}{p-1}}\\
 \leq &\;  \underset{S_{2^{k}}\,}{\,osc\,u}+2C_{0}C^{k},
\end{split}
\end{equation*}
and then, by \eqref{oscb1},
\begin{equation}\label{oscS|x|-new1}
\begin{split}
\underset{S_{|x|}\,}{\,osc\,u}&\leq \;\left( \underset{S_{1}\,}{\,osc\,u}+ \frac{2K}{\;C-2^{-\frac{\epsilon}{p-1}}\;}+ 2C_{0} \right)\;C^{k}
\end{split}
\end{equation}
Now \eqref{|x|} also gives
\begin{equation*}
\log|x|\,\le\,(k+1)\log2\,,\,\,\,\,\text{ hence }\,\,\,\,\, k\,\ge\,\frac{\log|x|}{\log2}-1
\end{equation*}
Therefore, as $C<1$, we have
\begin{equation*}
C^k\,\le\,\frac{C^{\frac{\log|x|}{\log2}}}{C}\,=\,\frac{\big(\,e^{\log C}\,\big)^{\frac{\log|x|}{\log2}}}{C}\,=\,\frac{|x|^{\frac{\log C}{\log2}}}{C}
\end{equation*}
and then, by \eqref{oscS|x|-new1}, if follows
\begin{equation}\label{oscS|x|-new2}
\begin{split}
\underset{S_{|x|}\,}{\,osc\,u}\,&\leq \,\frac{1}{C}\left( \underset{S_{1}\,}{\,osc\,u}+ \frac{2K}{\;C-2^{-\frac{\epsilon}{p-1}}\;}+ 2C_{0} \right)\,|x|^{\frac{\log C}{\log2}}
\end{split}
\end{equation}
\hfill$\square$

\section{Application to other problems}

As a consequence of Theorem \ref{TeorPrinc}, we have the following generalization:

\begin{Cor}
Let $u\in C^{1}(\mathbb{R}^{n}\backslash B_1)$ be a bounded weak solution of
\begin{equation*}
-{\rm div}\big(|\nabla u|^{p-2}A(|\nabla u|)\nabla u\big)\,=\,g(x,u)
\label{g-nonhomogeneousProblem}
\end{equation*}
in $\mathbb{R}^{n}\setminus B_1$,  where $A$ satisfies \eqref{A_conditions2}. Suppose that $|g(x,t)| \leq h_1(x) h_2(t)$, where $h_2 \in L^{\infty} (\mathbb{R})$.
Then the following hold:
\\[5pt] $(a)$ for $p >1$ the limit of $u$ at infinity exists provided $h_1$ satisfies \eqref{f}; \\
(In this case, assuming that $u$ is only bounded from above (or from below), then either $u$ converges at infinity or $\lim_{|x| \to +\infty} u(x) = -\infty \; (\text{or} +\infty)$;)
\\[5pt] $(b)$ for $1< p< n$ the limit of $u$ at infinity exists if, more generally, $h_1$ satisfies
\eqref{Lr}, \eqref{Ltheta} and \eqref{Kgoes0};
\\[5pt] $(c)$ for $p>n$, if $h_1$ satisfies \eqref{f}, there are positive constants $C, \beta$ such that
  $$ \big|\,\underset{\,|x|\rightarrow\infty}{\lim u}-u(x)\,\big| < C|x|^{-\beta} \quad \text{ for all } \quad |x| \quad {\rm large.} $$
Moreover, for $p\ge n$ these results also hold for unbounded weak solutions $u$ provided we assume that
\begin{equation}
  \lim_{x \to \infty} \frac{|u(x)|}{|x|^{\alpha}}=0 \quad {\rm where} \quad \alpha=\frac{p-n}{p-1} \quad \text{ if } \quad p > n
\label{alpha-subgrowthCor}
\end{equation}
or
\begin{equation}
  \lim_{x \to \infty} \frac{|u(x)|}{\log|x|}=0 \quad \text{ if } \quad p=n.
\label{alpha-subgrowth2Cor}
\end{equation} \\
\label{TeorPrincCor}
\end{Cor}
\begin{proof}
For each case, the result follows considering $f(x):= g(x,u(x))$ and observing that $f$ satisfies
the conditions of the respective case of Theorem \ref{TeorPrinc}.
\end{proof}

\begin{remark}
In this corollary, replacing $h_2 \in L^{\infty} (\mathbb{R})$ by $h_2 \in L^{\infty}_{loc} (\mathbb{R})$ the result still holds, provided that we assume that $u$ is always bounded. That is, the assumption that $u$ is only bounded from above or from below in the item (a) or the assumption that $u$ satisfies \eqref{alpha-subgrowthCor} or \eqref{alpha-subgrowth2Cor} are not sufficient.

Observe that, in particular, we can apply this corollary to the bounded weak solutions of the ``eigenvalue problem"
$$ -{\rm div}\big(|\nabla u|^{p-2}A(|\nabla u|)\nabla u\big) = \sum_{i=1}^k V_i(x)\, t \, |t|^{p_i-2} + h(x) \quad \text{in} \quad \mathbb{R}^n \backslash B_1, $$
for $p_i > 1$, where $V_i(x)$ and $h(x)$ satisfy \eqref{f} or satisfy \eqref{Lr}, \eqref{Ltheta} and \eqref{Kgoes0} for $i =1,2, \dots, k$.
This result is related to those obtained in \cite{FrP1} and \cite{FrP2}.
\end{remark}

\newpage

\noindent Leonardo Prange Bonorino\newline\noindent Universidade Federal do Rio
Grande do Sul\newline\noindent Brazil\newline\noindent leonardo.bonorino@ufrgs.br

\medskip

\noindent Lucas Pinto Dutra\newline\noindent Instituto Federal de Educa\c{c}\~ao, Ci\^encia e Tecnologia do Rio Grande do Sul - Campus Caxias do Sul\newline\noindent Brazil\newline\noindent lucas.dutra@caxias.ifrs.edu.br

\medskip

\noindent Filipe Jung dos Santos\newline\noindent Universidade Federal do Rio
Grande do Sul\newline\noindent Brazil\newline\noindent jung.filipe@gmail.com

\end{document}